\documentclass{amsart}

\usepackage{tikz}
\usetikzlibrary{cd}
\usepackage{amsthm} 
\usepackage{amsfonts}
\usepackage{amssymb,latexsym, mathrsfs}
\usepackage{amsmath, mathtools}  
\usepackage{verbatim,enumerate}
\usepackage{xcolor}
\usepackage{hyperref}
\usepackage[utf8]{inputenc}
\usepackage[noabbrev,capitalise]{cleveref}

\numberwithin{equation}{section}

\theoremstyle{plain}
\newtheorem{theorem}{Theorem}[section]
\newtheorem{ntheorem}{Theorem}
\newtheorem*{theorem*}{Theorem}

\newtheorem{lemma}[theorem]{Lemma}
\newtheorem{proposition}[theorem]{Proposition}
\newtheorem{corollary}[theorem]{Corollary}

\theoremstyle{definition} 
\newtheorem{definition}[theorem]{Definition}
\newtheorem{example}[theorem]{Example}
\newtheorem{remark}[theorem]{Remark}

\DeclareMathOperator{\En}{End}
\DeclareMathOperator{\Aut}{Aut}

\DeclareMathOperator{\Mod}{mod}
\DeclareMathOperator{\proj}{proj}
\DeclareMathOperator{\inj}{inj}
\DeclareMathOperator{\Id}{Id}
\DeclareMathOperator{\Ker}{Ker}

\DeclareMathOperator{\Imm}{Im}
\DeclareMathOperator{\Hom}{Hom}

\DeclareMathOperator{\rad}{rad}
\DeclareMathOperator{\Cone}{Cone}
\DeclareMathOperator{\CCone}{Cocone}
\DeclareMathOperator{\EE}{\mathbb{E}}
\DeclareMathOperator{\add}{add}
\DeclareMathOperator{\Fac}{Fac}
\DeclareMathOperator{\thi}{thick}
\DeclareMathOperator{\Filt}{Filt}

\def\K{\ensuremath{\mathcal{K} }}
\newcommand{\X}{\mathcal{X}}
\newcommand{\Y}{\mathcal{Y}}
\newcommand{\Tt}{\mathscr{T}}
\newcommand{\Ww}{\mathscr{W}}
\newcommand{\T}{\mathcal{T}}
\newcommand{\F}{\mathcal{F}}
\newcommand{\U}{\mathcal{U}}
\newcommand{\W}{\mathcal{W}}
\newcommand{\C}{\mathcal{C}}
\newcommand{\D}{\mathcal{D}}

\DeclareMathOperator{\ctor}{cotor}
\DeclareMathOperator{\cctor}{c-cotor}
\DeclareMathOperator{\silt}{silt}
\DeclareMathOperator{\tautilt}{s\tau-tilt}
\DeclareMathOperator{\wide}{wide}
\DeclareMathOperator{\fwide}{l-wide}
\DeclareMathOperator{\ftor}{f-tors}
\DeclareMathOperator{\tor}{tors}
\DeclareMathOperator{\res}{res}
\DeclareMathOperator{\fres}{f-res}
\DeclareMathOperator{\fthi}{inj-thick}
\usepackage{array}
\usetikzlibrary{calc}
\usetikzlibrary{shapes}
\usepackage{tabularx}
\usepackage{longtable}
\pgfdeclarelayer{bg}
\pgfsetlayers{bg,main}

\begin{document}
	
 \title{On thick subcategories of the category of projective presentations}
 \author{Monica Garcia}
 \address{Laboratoire de Mathématiques de Versailles, UVSQ, CNRS, Université Paris-Saclay. Office 3305, Bâtiment Fermat, 45 Avenue des États-Unis, 78035 Versailles CEDEX, France}
 \email{monica.garcia@uvsq.fr}
%
 \begin{abstract}
  We study thick subcategories of the category of 2-term complexes of projective modules over an associative algebra. We show that those thick subcategories that have enough injectives are in explicit bijection with 2-term silting complexes and complete cotorsion pairs. We also provide a bijection with left finite wide subcategories of the module category and prove that all these maps are compatible with previously known correspondences. We discuss possible applications to stability conditions. 
  \smallskip
  
\noindent \textbf{Keywords.} Projective presentation, thick subcategory, cotorsion pair, silting object, wide subcategory, semistability. 
 \end{abstract}
  
 \maketitle

 \tableofcontents

\section{Introduction}
In their seminal paper on $\tau$-tilting theory \cite{adachi2014tilting}, T.~Adachi, O.~Iyama and I.~Reiten studied the relationship between several classes of objects, namely, support $\tau$-tilting modules, 2-term silting complexes, and functorially finite torsion classes. Since then, driven by applications to cluster theory \cite{buan2006tilting, amiot2009cluster} and stability conditions \cite{asai2020semibricks, brustle2022stability}, among others, many classes of objects have been added to this list. In the category $\Mod \Lambda$ of finitely generated modules over an associative algebra $\Lambda$ (see \cref{preliminaries} for our assumptions), this list includes
\begin{itemize}
	\item[$-$] support $\tau$-tilting objects (\cite{ingalls2009noncrossing, adachi2014tilting, demonet2017lattice}, ...),
	\item[$-$] the lattice of torsion pairs (\cite{dickson1966torsion, ingalls2009noncrossing, iyama2015lattice}, ...),
	\item[$-$] the poset of wide subcategories (\cite{hovey2001classifying, buan2021category, buan2021tau}, ...),
\end{itemize}
to name a few. These classes of objects are known to be intimately related to each other, see for instance \cite{ingalls2009noncrossing, demonet2019tilting, asai2022wide}. Another important category associated to an algebra $\Lambda$ is the category of 2-term complexes of finitely generated projective modules, which we will denote by $\K^{[-1,0]}(\proj \Lambda)$. This category plays an essential role, for instance, in the categorification of g-vectors of cluster algebras \cite{bazier2018abhy, padrol2019associahedra}, and appears naturally as the extended co-heart of particular co-t-structures on the homotopy category $\K^b(\proj \Lambda)$ \cite{pauksztello2020co,buan2021weak}. Many of the objects in $\Mod \Lambda$ we have alluded to have a ``mirror" analog in $\K^{[-1,0]}(\proj \Lambda)$, namely
\begin{itemize}
	\item[$-$] 2-term silting complexes, as the analog of support $\tau$-tilting modules (\cite{keller1988aisles, adachi2014tilting, derksen2015general}, ...);
	\item[$-$] cotorsion pairs, as the analog of torsion pairs (\cite{salce1979cotorsion, nakaoka2019extriangulated, pauksztello2020co}, ...).
\end{itemize}
Like their counterparts in $\Mod \Lambda$, they have been shown be in one-to-one correspondence, see \cite{pauksztello2020co, adachi2022hereditary}. 
In this paper, we introduce to this last list the class of thick subcategories of $\K^{[-1,0]}(\proj \Lambda)$ and claim that they form the ``mirror" analog of wide subcategories in $\Mod \Lambda$. We explicitly relate them to cotorsion pairs and 2-term silting complexes. Before stating our results, let us recall the explicit bijections between some of the classes of objects listed above. 

\begin{ntheorem}{\cite[Theorem 1.1]{ingalls2009noncrossing}\cite[Theorem 0.5]{adachi2014tilting}\cite[Theorem 30]{marks2017torsion}\cite[Theorem 1.2]{yurikusa2018wide} \cite[Theorem 1.1]{brustle2019wall}} \label{theo-mod}
	Let $\Lambda$ be a finite-dimensional algebra over a field. There are explicit bijections between the sets of
	\begin{enumerate}
		\item Isomorphism classes of basic support $\tau$-tilting modules in $\Mod \Lambda$.
		\item Functorially finite torsion pairs in $\Mod \Lambda$.
		\item Left finite wide subcategories of $\Mod \Lambda$.
		\item Left finite semistable subcategories of $\Mod \Lambda$. 
	\end{enumerate}
\end{ntheorem}
The bijection from (1) to (2) takes a support $\tau$-tilting module $(M,P)$ to the torsion pair $\vartheta(M) = (\Fac(M), M^\perp)$. Here, $\Fac(M) = \{ N \in \Mod \Lambda \ | \ \exists \ M' \rightarrow N \rightarrow 0 \text{ s.e.s. with } M' \in \add(M)\} = {}^\perp(\tau M) \cap P^\perp$. The bijection from (2) to (3) is given by the map $\alpha(\T) = \{ M \in \T \ | \ \forall(g : N \rightarrow M) \in \T,\  \ker(g) \in \T\}$ for any torsion pair $(\T,\F)$. Finally, the bijection from (1) to (4) is obtained by proving that $\alpha (\Fac (M)) = M_\rho^\perp \cap \Fac(M) = \Ww_{g^{M\rho} - g^P}$, where the later is the semistable subcategory associated to the \textit{g-vector} of $M_\rho$ minus the g-vector of $P$. Here, $M_\rho$ is the basic module such that $\add(M_\rho) = \add(M_1)$ for $M_1$ satisfying that 
\begin{equation}\label{approx-tilt}
	\Lambda \rightarrow M_0 \rightarrow M_1 \rightarrow 0
\end{equation}
is a minimal left $M$-approximation of $\Lambda$. 

Both support $\tau$-tilting modules and torsion pairs turn out to have ``mirror" analogs in the extriangulated category $\K^{[-1,0]} (\proj \Lambda)$.
\begin{ntheorem}\cite[Theorem 3.2]{adachi2014tilting}\label{silt-tilt}
	Let $\Lambda$ be a finite-dimensional algebra over a field. There exists an explicit bijection between
	\begin{enumerate}
		\item[(1)] Isomorphism classes of basic support $\tau$-tilting modules in $\Mod \Lambda$.
		\item[(i)] Isomorphism classes of basic silting objects in $\K^{[-1,0]}(\proj \Lambda)$.
	\end{enumerate}
\end{ntheorem}
This bijection takes any 2-term silting object $U$ and sends it to $H^0(U)$. Its inverse sends a support $\tau$-tilting module $(M,P)$ to \begin{tikzcd}[cramped, sep=small] 
		P^{-1}_M \oplus P \arrow[d, "{(f,0)}"] \\ P^0_M \end{tikzcd}, where \begin{tikzcd}[cramped, sep=small] P^{-1}_M \arrow[d, "f"] \\ P^0_M \end{tikzcd} is a minimal projective presentation of $M$. 

Let $U \in \K^{[-1,0]}(\proj \Lambda)$ be a basic silting object. Inside of $\K^{[-1,0]}(\proj \Lambda)$ there is a conflation 
\begin{equation}\label{approx-silt}
	\Lambda \xrightarrow{f} U_0 \rightarrow U_1 \overset{g}{\dashrightarrow} \Lambda[1]
\end{equation}
where $f$ is a minimal left $U$-approximation of $\Lambda$ and $g$ a minimal right $U$-approximation of $\Lambda[1]$. We will denote by $U_\lambda$ and $U_\rho$ the direct summands of $U$ satisfying that $U = U_\lambda \oplus U_\rho$, $U_0 \in \add(U_\lambda)$ and $U_1 \in \add(U_\rho)$, respectively. We note that the short exact sequence~(\ref{approx-tilt}) can be obtained by applying the cohomological functor $H^*(-)$ to the triangle~(\ref{approx-silt}).

To simplify notation, we let $\K_\Lambda = \K^{[-1,0]}(\proj \Lambda)$. The following is an application of a more general theorem concerning cotorsion pairs in triangulated categories: 
\begin{ntheorem}\cite[Theorem 3.6]{pauksztello2020co}\label{C} There is a well defined map $$\Phi : \ctor \K_\Lambda \rightarrow \tor \Lambda$$ between the set $\ctor \K_\Lambda$ of cotorsion pairs in $\K^{[-1,0]}(\proj \Lambda)$ and the set $\tor \Lambda$ of torsion pairs  in $\Mod \Lambda$. This map restricts to a bijection between
	\begin{enumerate}
		\item[(2)] Functorially finite torsion pairs in $\Mod \Lambda$.
		\item[(ii)] Complete cotorsion pairs in $\K^{[-1,0]}(\proj \Lambda)$.
	\end{enumerate}	
\end{ntheorem}
In~\cite{adachi2022hereditary}, the authors constructed an explicit bijection $\Psi$ between cotorsion pairs in a general extriangulated category $\K$ and its silting subcategories. When $\K = \K_\Lambda$, this gives
\begin{ntheorem}\cite[Theorem 5.7]{adachi2022hereditary}\label{D}
	The following sets are in one-to-one correspondence: 
	\begin{enumerate}[(i)]
		\item Isomorphism classes of basic silting objects in $\K^{[-1,0]}(\proj \Lambda)$.
		\item Complete cotorsion pairs in $\K^{[-1,0]}(\proj \Lambda)$.
	\end{enumerate}
This correspondance takes a complete cotorsion pair $(\X, \Y)$ and sends it to $\Psi((\X, \Y)) = U$, where $U$ is a basic additive generator of the silting category $\X \cap \Y$. 
\end{ntheorem}
The notion of thick subcategory of an extriangulated category was first introduced in~\cite{nakaoka2022localization} in order to generalize the notion of localization of both exact and triangulated categories. In this paper, we complete the ``mirror" version of \cref{theo-mod} in $\K^{[-1,0]}(\proj \Lambda)$ using thick subcategories. This is done in such a way that all the bijections appearing in the previous theorems commute with each other.  

\begin{theorem*}[\textbf{\ref{cotorsion-thick}}] Let $\Lambda$ be a finite-dimensional algebra over a field, and let $\K_\Lambda =\K^{[-1,0]}(\proj \Lambda)$. There exist well defined maps
		\begin{center}
			\begin{tikzcd}
				\ctor \K_\Lambda \arrow[r, shift left=.75ex, "\beta"] & \thi \K_\Lambda \arrow[l, shift left=.75ex, "\iota"]
			\end{tikzcd}
		\end{center}
	such that when restricted to the set $\cctor \K_\Lambda$ of complete cotorsion pairs and the set $\fthi \K_\Lambda$ of thick subcategories with enough injectives, they are inverse of each other.
\end{theorem*}

\begin{theorem*}[\textbf{\ref{theo-thick-wide}}]
	Let $\Lambda$ be a finite-dimensional algebra over a field and take $\K_\Lambda$ as before. There exist inclusion-reversing maps
		\begin{center}
			\begin{tikzcd}
				\wide \Lambda \arrow[r, shift left=.75ex, "\Tt"] & \thi \K_\Lambda \arrow[l, shift left=.75ex, "\Ww"] 
			\end{tikzcd}
		\end{center}
	such that, when restricted to thick subcategories with enough injectives and the set $\fwide \Lambda$ of left finite wide subcategories, they make the following diagram commute
	\begin{center}
		\begin{tikzcd}[cramped,column sep=large, row sep=scriptsize]
			& & \silt \K_\Lambda \arrow[dr, "\thi(U_\rho)"] & & \\
			&  \cctor\K_\Lambda \arrow[ur, "\Psi"]  \arrow[rr, "\beta", crossing over]& & \fthi \K_\Lambda& \\ 
			\mathclap{\K^{[-1,0]}(\proj \Lambda)} & & & &  \\
			{} \arrow[rrrr, dashed, no head, crossing over]& & & & {} \\
			& & & & \mathclap{\Mod \Lambda} \\
			& \ftor\Lambda \arrow[from= uuuu, "\Phi", crossing over, crossing over clearance=7ex] & & \fwide\Lambda \arrow[from= ll, "\alpha"] \arrow[from= uuuu, crossing over, crossing over clearance=7ex,"\Ww"] & 
		\end{tikzcd}
	\end{center} 
	In particular, $\Ww$ and $U \in \silt \K_\Lambda \mapsto \thi(U_\rho) \in \fthi \K_\Lambda$ are bijective.
\end{theorem*}
Putting all previous theorems together we get the following result.

\begin{corollary}\label{main-coro}
	There are explicit bijections between
	\begin{enumerate}[(i)]
		\item Isomorphism classes of basic silting objects in $\K^{[-1,0]}(\proj \Lambda)$.
		\item Complete cotorsion pairs in $\K^{[-1,0]}(\proj \Lambda)$.
		\item Thick subcategories in $\K^{[-1,0]}(\proj \Lambda)$ with enough injectives.
	\end{enumerate}
These bijections are compatible with those in Theorems~\ref{theo-mod}, \ref{silt-tilt}, \ref{C}, and \ref{D}. In other words, the following diagram commutes
\begin{center}
	\begin{tikzcd}[cramped,column sep=large, row sep=scriptsize]
		& & \silt \K_\Lambda \arrow[ddddd, "H^0" {yshift= 5ex}] \arrow[dr, "\thi(U_\rho)"] & & \\
		&  \cctor\K_\Lambda \arrow[ur, "\Psi"]  \arrow[rr, "\beta" {xshift= 3ex}, crossing over]& & \fthi \K_\Lambda & \\ 
		\mathclap{\K^{[-1,0]}(\proj \Lambda)} & & & &  \\
		{} \arrow[rrrr, dashed, no head, crossing over]& & & & {} \\
		& & & & \mathclap{\Mod \Lambda} \\
		& & \tautilt\Lambda \arrow[dl, "\Fac"] \arrow[dr]& & \\
		& \ftor\Lambda \arrow[from= uuuuu, "\Phi", crossing over, crossing over clearance=7ex] & & \fwide\Lambda \arrow[from= ll, "\alpha"] \arrow[from= uuuuu, crossing over, crossing over clearance=7ex,"\Ww"] & 
	\end{tikzcd}
\end{center}
\end{corollary}
Our work is motivated by the possible extension of the theory of stability conditions to $\K^{[-1,0]}(\proj \Lambda)$. 
We use as inspiration the description of certain wide subcategories of $\Mod \Lambda$ as subcategories of semistable modules for a particular collection of stability conditions \cite{schofield1991semi, King, yurikusa2018wide, brustle2019wall} for the introduction of the notion of $M$-stability for two-term complexes with $M \in \Mod \Lambda$ (see \cref{defMsemistab1} and \ref{defMsemistab2}). In \cref{s_geometry} we compare this representation theoretical notion of semistability with the one arising from Geometric Invariant Theory (\cref{geometricss1}, \cref{gemetricss2}), and a numerical one based on King's semistability for modules (\cref{numericalss}). We show that both geometric and $M$-stability imply numerical semistability (see \cref{naive} and \ref{Mimplies[M]}), but that the converse does not hold in general. This result, together with \cref{main-coro}, and \cref{wide-thick}, strongly suggests that $M$-stability is a notion worth of further study. 
\section*{Acknowledgments}
I am grateful to P.-G. Plamondon for the many constructive discussions, his availability, and his support throughout the development of my Ph.D. thesis, from which this article emanates. I would also like to thank Yann Palu for his helpful comments, in particular about \cref{lfinitethi}. I grateful to Calin Chindris, Charles Paquette and an anonymous referee for their comments and suggestions for improving \cref{s_geometry}. Lastly, I thank the ISM Discovery School on Mutations organizers, where I acquired crucial tools for the progress of this work. 

 
\section{Preliminaries}\label{preliminaries}
We fix $\Bbbk$ an algebraically closed field of characteristic zero. Let $\Lambda$ be a finite-dimensional $\Bbbk$-algebra. We will mostly consider the case where $\Lambda \cong \Bbbk Q/ I$ with $Q = (Q_0, Q_1)$ a finite quiver and $I$ an admissible ideal of the path algebra $\Bbbk Q$. We write $S_i$ for the $i$-th simple module associated to a vertex $i \in Q_0$ with corresponding projective cover $P_i \twoheadrightarrow S_i$. Recall that we can associate to $\Lambda$ the triangulated category $\D^b(\Mod \Lambda)$ of bounded complexes of finite-dimensional modules, where $X[1]$ will denote the shift of any $X \in \mathcal{D}^b(\Mod\Lambda)$. We consider as well the category of bounded complexes of projective modules $\K^b(\Mod \Lambda)$, and we will denote by $\K^{[a,b]} (\proj \Lambda)$ with $a \leq b \in \mathbb{Z}$ the extension-closed subcategory of $\K^b(\Mod \Lambda)$ of complexes concentrated between degrees $a$ and $b$. In particular, we will study the category $\K_\Lambda = \mathcal{K}^{[-1,0]}(\proj\Lambda)$, which we view as the category of morphisms between projective modules up to homotopy. 

Let $n = |Q_0|$. Denote by $K_0(\Mod \Lambda) \cong K_0(\D^b(\Mod \Lambda))$ the Grothendieck group of $\Mod\Lambda$, which is canonically isomorphic to the lattice $\bigoplus_{i = 1}^n \mathbb{Z} [S_i]$. Similarly, we consider its dual $K_0(\proj\Lambda) \simeq K_0(\K^b(\proj \Lambda)) \cong K_0(\K_\Lambda) \cong \bigoplus_{i = 1}^n \mathbb{Z} [P_i]$. The Euler form associated to these two groups is given by 
\begin{align*}
	\langle -, - \rangle : K_0(\proj\Lambda) \times K_0(\Mod\Lambda)  & \longrightarrow \mathbb{Z} \\
	([P_j], [S_i]) \mapsto \langle  [P_j], [S_i] \rangle & = \begin{cases}
		1 & i = j \\
		0 &  i \neq j \ .
	\end{cases}
\end{align*}
In particular, for every $M \in \Mod\Lambda$ and $X =$ \begin{tikzcd}[cramped, sep=small]X^{-1} \arrow[d, "x"] \\ X^0 \end{tikzcd} $\in \K_\Lambda$, this pairing is given by \begin{gather*}\langle [X], [M] \rangle = \langle  [X^0]-[X^{-1}], [M] \rangle = \dim_\Bbbk(\Hom_{\Lambda}(X^0, M)) - \dim_\Bbbk(\Hom_{\Lambda}(X^{-1}, M)). \end{gather*}

\subsection{Extriangulated categories}
Extriangulated categories were introduced by Nakaoka and Palu in \cite{nakaoka2019extriangulated} as a way to generalize both triangulated and exact categories. Since $\K_\Lambda = \mathcal{K}^{[-1,0]}(\proj \Lambda)$ is an extension-closed subcategory of  the triangulated category $\mathcal{D}^b(\Mod\Lambda)$, then it is extriangulated. In this setting, the bifunctor associated to $\K_\Lambda$ that gives its extriangulated structure is given by $\mathbb{E}_{\K_\Lambda}(X,Y) = \Hom_{\mathcal{D}^b(\Mod\Lambda)}(X,Y[1])$. A \textbf{conflation} in an extriangulated category is the generalization of what would be a triangle in a triangulated category, or an exact sequence in an exact category. We restrict this notion to our setting.

\begin{definition}
	A sequence of morphisms \begin{tikzcd}[cramped, sep=small] 
		X \arrow[r, "f", tail] & Y \arrow[r, "g", twoheadrightarrow] & Z 
	\end{tikzcd} in $\K_\Lambda$ is a \textbf{conflation} if there exists a map $h : Z \rightarrow X[1]$ such that $(f,g,h)$ is a triangle in $\mathcal{D}^b(\Mod \Lambda)$. In this scenario, $f$ is said to be an \textbf{inflation} and $g$ a \textbf{deflation}.
\end{definition}

\begin{remark}\label{definf}
Recall that the category $\C_\Lambda = \C^{[-1,0]} (\proj \Lambda)$ of morphism between projective modules is exact, thus, extriangulated. Indeed, its conflations are given by those sequences \begin{tikzcd}[cramped, sep=small]
		A \arrow[r, "u", tail] & B \arrow[r, "v", twoheadrightarrow] & C 
	\end{tikzcd} of objects in $\C_\Lambda$ such that the sequences $0 \rightarrow A^{i} \xrightarrow{u^{i}} B^i \xrightarrow{v^i} C^i \rightarrow 0$ are exact in $\Mod \Lambda$ for $i = -1,0$. Let \begin{tikzcd}[cramped, sep=small] 
 	X \arrow[r, "f", tail] & Y \arrow[r, "g", twoheadrightarrow] & Z 
 \end{tikzcd} be a conflation in $\K_\Lambda$ and fix representatives $x, y, z$ of the differentials of $X, Y, Z$ respectively. Chose as well $h : Z^{-1} \rightarrow X^0$, a representative of the morphism $Z \dashrightarrow X[1]$ associated to our conflation. Then we must have an isomorphism in $\K_\Lambda$ $$Y \cong \CCone(Z^{-1} \dashrightarrow X[1])[-1] \cong \text{\begin{tikzcd}[cramped]  X^{-1} \oplus Z^{-1} \arrow[d, "{\begin{psmallmatrix} x & h \\ 0 & z \end{psmallmatrix}}"]\\X^0\oplus Z^0  \end{tikzcd}}.$$
 
Since $\K_\Lambda$ is equivalent to the (extriangulated) quotient (see \cite{nakaoka2022localization}) $\C_\Lambda/\{ P \xrightarrow{f} Q \ | \  f \text{ isomorphism} \}$, if we choose a minimal representative $y$  of $Y$, that is, such that it satisfies that $y \ncong \begin{pmatrix}
		y' & 0 \\ 0 & \Id_Q
	\end{pmatrix}$ for all $0 \neq Q  \in \proj \Lambda$, then there exists $P \in \proj \Lambda$ and a diagram
	\begin{center}
		\begin{tikzcd}[ampersand replacement=\&, row sep = large]
			Y^{-1} \oplus P \arrow[r, "\simeq"] \arrow[d, "{\begin{pmatrix} y & 0 \\ 0 & \Id_P \end{pmatrix}}"] \& X^{-1}\oplus Z^{-1} \arrow[d, "{\begin{pmatrix} x & h \\ 0 & z \end{pmatrix}}"] \\
			Y^0 \oplus P \arrow[r, "\simeq"] \&  X^{0}\oplus Z^{0} \end{tikzcd}
	\end{center}
	that is commutative inside $\Mod \Lambda$. That is, the obtained sequence $X \rightarrowtail Y \oplus \text{\begin{tikzcd}[cramped, sep=tiny]  P \arrow[d, equal]\\ P \end{tikzcd}} \twoheadrightarrow Z$ is a conflation inside of $\C_\Lambda$, which represents the conflation $(f,g)$ in $\K_\Lambda$.
\end{remark}

Besides being Krull-Schmidt and $\Hom$-finite, the category $\K_\Lambda$ satisfies 4 key properties:

\begin{enumerate}[(a)]
	\item $\K_\Lambda$ is \textbf{hereditary}, that is $\EE^i(X,Y) = \Hom_{\K_\Lambda}(X,Y[i]) = 0$ $\forall i \geq 2$ and all $X, Y \in \K^{[-1,0]}(\proj \Lambda)$.
	\item It has \textbf{enough projectives}, which are given by the complexes of the form $P : = $\begin{tikzcd}[cramped, sep=small] 
		0 \arrow[d] \\P \end{tikzcd} where $P \in \proj \Lambda$.
	\item It has \textbf{enough injectives}, which are all objects isomorphic to complexes $P[1] := $\begin{tikzcd}[cramped, sep=small] 
		P \arrow[d] \\ 0 \end{tikzcd} where $P \in \proj \Lambda$.
	\item It satisfies \textbf{WIC}: If $h= f \circ g$ is an inflation, then so is $g$. Dually, if $h$ is a deflation, then so is $f$. 
\end{enumerate}

\subsection{Cotorsion pairs and $\tau$-tilting theory}
In this section we recall some results and tools used in the study of cotorsion pairs and silting subcategories in a general extriangulated category $\K$. Most of these results first appeared in the context of extriangulated categories in~\cite{adachi2022hereditary}. We will often apply these results to the category $\K_\Lambda$. From now on, we suppose that all our subcategories are full and stable under isomorphisms. 

\begin{definition}
	Let $\K$ be an extriangulated category and $\X$ and $\Y$ two subcategories of $\K$. We denote by $\Cone(\X, \Y)$ the full subcategory whose objects are those $Z$ such that there exists a conflation $X \rightarrowtail Y \twoheadrightarrow Z$, where $X \in \X$ and $Y \in \Y$. Dually, we say that $ Z \in \CCone(\X, \Y)$ if there exists a conflation $Z \rightarrowtail X \twoheadrightarrow Y$ with $X \in \X$ and $Y \in \Y$. We define as well
	\begin{enumerate}
		\item $\X^{\wedge}_{-1} = 0$, $\X^{\wedge}_m = \Cone(\X^{\wedge}_{m-1}, \X)$  $\forall m \in \mathbb{Z}_{\geq 0}$, and $\X^{\wedge} = \bigcup_{m \in \mathbb{Z}_{\geq 0}} \X^{\wedge}_m$.
		\item $\X^{\vee}_{-1} = 0$, $\X^{\vee}_m = \CCone(\X, \X^{\vee}_{m-1})$ $\forall m \in \mathbb{Z}_{\geq 0}$, and $\X^{\vee} = \bigcup_{m \in \mathbb{Z}_{\geq 0}} \X^{\vee}_m$.
	\end{enumerate}
\end{definition}

\begin{definition}\cite{nakaoka2022localization}
	Let $\mathcal{K}$ be an extriangulated category. We say that a subcategory $\mathcal{T} \subset \K$ is \textbf{thick}, if it is closed under direct summands and if for every conflation 
	$X \rightarrowtail Y \twoheadrightarrow Z$
	if two of the terms lie in $\mathcal{K}$, then the third does as well. That is, $\K$ is closed under extensions, cones and cocones. For all $\mathcal{C} \subset \K$, we denote by $\thi(\mathcal{C})$ the smallest thick subcategory that contains $\mathcal{C}$, and we write $\thi \K$ for the set of thick subcategories of $\K$. 
\end{definition}

\begin{definition}
	Let $\K$ be an extriangulated category. We say that a subcategory $\U\subset \K$ is \textbf{presilting} if it is closed under direct sums and summands and $\EE^i(\U, \U) = 0$ for all $i > 0$. We say that $\U$ is \textbf{silting} if $\thi(\U) = \K$. An object $U \in \K$ is (pre)silting if the category $\add(U)$ is. We denote by $\silt \K$ the set of isomorphism classes of basic silting objects in $\K$. 
\end{definition}

\begin{proposition}\cite[Lemma 5.3]{adachi2022hereditary}\label{uincv}
	Let $\mathcal{V} \subset \K$ be a silting subcategory. If $\U$ is a presilting subcategory with $\mathcal{V}  \subset \U$, then $\U = \mathcal{V}$.
\end{proposition}

\begin{proposition}~\cite[Proposition 5.4]{adachi2022hereditary}\label{siltinggenerator}
	Let $\K$ be an extriangulated category that contains a silting object. Then each silting category admits an additive generator. Moreover, if $\K$ is a Krull-Schmidt category, then $U \mapsto \add U$ gives a bijection between $\silt \K$ and the set of silting subcategories. 
\end{proposition}

\begin{proposition}\cite[Proposition 4.10]{adachi2022hereditary}\label{vee-wedge-desc}
	Let $\U$ be a presilting subcategory of $\K$. Then the following statements hold.
	\begin{enumerate}
		\item $\U^\vee$ is the smallest subcategory containing $U$ and closed under extensions, cocones and direct summands. Moreover, if $\U$ is closed under cones, then $\U^\vee = \thi(\U)$.
		\item $\U^\wedge$ is the smallest subcategory containing $U$ and closed under extensions, cones and direct summands. Moreover, if $\U$ is closed under cocones, then $\U^\wedge = \thi(\U)$.
	\end{enumerate}
\end{proposition}

\begin{definition}\cite[Definition 1.7]{pauksztello2020co}
	Let $\K$ be an extriangulated category. We say that a pair of subcategories $(\X, \Y)$ is a \textbf{cotorsion pair} if they are both full and additive and they satisfy 
	\begin{enumerate}
		\item $\EE(X, \Y) = 0$ if and only if $X \in \X$.
		\item  $\EE(\X, Y) = 0$ if and only if $Y \in \Y$.
	\end{enumerate}
In other words $\Y = \X^{\perp_1} = \{Y \in \K \ | \ \EE(X,Y) = 0 \ \forall \ X\in \X\}$ and $\X = {}^{\perp_1}\Y = \{X \in \K \ | \ \EE(X,Y) = 0 \ \forall \ Y\in \Y\}$. We denote by $\ctor\K$ the set of all cotorsion pairs in $\K$. We say that $(\X, \Y)$ is \textbf{complete} \cite[Definition 4.1]{nakaoka2019extriangulated}, if additionally $\K= \Cone(\Y,\X) = \CCone(\Y, \X)$. We denote by $\cctor \K \subset \ctor \K$ the subset of complete cotorsion pairs of $\K$. 
\end{definition}

\begin{remark}
	As we have noted before, when $\K = \K_\Lambda$ it is always true that $\EE^2(X,Y) = 0$ for all $X, Y \in \K$. In particular, $\EE^2(\X, \Y) = 0$ for all cotorsion pairs $(\X,\Y)$. We say that a cotorsion pair is \textbf{hereditary} \cite[Definition 4.1]{liu2020hereditary} when it satisfies this property. We remark as well that all projective objects must lie in $\X$, all injective objects belong to $\Y$ and since $\K_\Lambda = \Cone (\proj \K_\Lambda, \proj \K_\Lambda) = \CCone (\text{inj } \K_\Lambda, \text{inj }\K_\Lambda)$ we have that $\K_\Lambda= \X^{\wedge} = \Y^{\vee}$. In a general extriangulated category $\K$, we say that $(\X, \Y)$ is \textbf{bounded} \cite{adachi2022hereditary} if it satisfies that $\K= \X^{\wedge} = \Y^{\vee}$.
\end{remark}

Recall that a \textbf{torsion pair} in $\Mod \Lambda$ is a pair of subcategories $(\T, \F)$ closed under extensions such that 
\begin{enumerate}
	\item $\T$ is closed under factor modules.
	\item $\F$ is closed under submodules.
	\item $\F = \T^\perp$ and $\T = {}^\perp\F$. 
\end{enumerate}
We say that a torsion pair is \textbf{functorially finite} if 
every $\Lambda$-module admits both a right and a left $\T$-approximation. We denote by $\ftor \Lambda \subset \tor \Lambda$ the subset of functorially finite torsion classes in $\Mod \Lambda$. 

\begin{theorem}\cite[Theorem 3.6]{pauksztello2020co}\label{cotorsiontorsion}
	Let $\K_\Lambda= \K^{[-1, 0]}(\proj \Lambda)$. There are well defined maps:
	\begin{center}
		\begin{tikzcd}[column sep=large]
			\ctor \K_\Lambda \arrow[r, "\Phi"] & \tor \Lambda \\
			\cctor \K_\Lambda \arrow[u, phantom, sloped, "\subseteq"] & \ftor \Lambda \arrow[l, "\Theta"] \arrow[u, phantom, sloped, "\subseteq"]
		\end{tikzcd}
	\end{center}
given by 
	$$\Phi((\X, \Y)) = (H^0(\Y), H^0(\Y)^{\perp})$$
and
	$$\Theta(\T, \F) = ( {}^{\perp_1}\mathcal{Z}, \mathcal{Z})$$
where $\mathcal{Z} = \left(H^0\right)^{-1}(\T)$. Moreover, $\Theta$ and $\Phi$ are inverse to each other when restricted to $\cctor \K$ and $\ftor \Lambda$. 
\end{theorem}

The following is an application of Theorem 5.7 in~\cite{adachi2022hereditary} to the extriangulated category $\K_\Lambda$. The statement follows from the fact that, in $\K_\Lambda$, all complete cotorsion pairs are hereditary and bounded. The original theorem gives a bijection between $\cctor \K_\Lambda$ and the set of its silting subcategories. Since $\K_\Lambda$ is Krull-Schmidt, \cref{siltinggenerator} allows us to state the bijection in terms of $\silt \K_\Lambda$. 

\begin{theorem}\cite[Theorem 5.7]{adachi2022hereditary}\label{cotorsionsilting} There is a one-to-one correspondence
	\begin{center}
		\begin{tikzcd}
			\cctor \K_\Lambda \arrow[r, "\Psi", shift left=.75ex] & \silt \K_\Lambda \arrow[l, "\Xi", shift left=.75ex]
		\end{tikzcd}
	\end{center}
	given by the maps 
	$$\Xi(U) = (\add(U)^{\vee}, \add(U)^{\wedge})$$
	and
	$$\Psi((\X, \Y)) = U$$
	where $U$ is a basic additive generator of the silting category $\X \cap \Y$. 
\end{theorem}

\begin{definition}
	Let $\K$ be an extriangulated category and let $\X \subset \K$. We say that $\X$ is a \textbf{resolving subcategory} of $\K$ if $\K= \Cone (\K, \X)$ and it is closed under extensions, cocones and direct summands. Moreover, we say that $\X$ is \textbf{contravariantly finite} if every object in $\K$ admits a right $\X$-approximation. We write $\fres \K \subset \res \K$ for the sets of contravariantly finite resolving subcategories and resolving subcategories of $\K$, respectively. 
\end{definition}
\begin{remark}
	When $\K$ has enough projectives, we can swap the condition of $\X$ satisfying $\K = \Cone(\K, \X)$ in the previous definition by $\X$ containing all projective objects. Indeed, if $\proj \K \subset \X$, since all objects $C$ admit a deflation $P_C \twoheadrightarrow C$ with $P_C \in \proj \K$, then $\K \subset \Cone(\K, \proj \K) \subset \Cone(\K, \X)$. Reciprocally, if $\K = \Cone(\K, \X)$, then for all $P \in \proj \Lambda$ there is $X_P \in \X$ and a conflation $Y \rightarrowtail X_P \twoheadrightarrow P$. But since $P$ is projective, this conflation must split and $P \in \X$, as $\X$ is closed under direct summands.
	
	Suppose now that $\K = \K_\Lambda$ and let $\X$  be a resolving and contravariantly finite subcategory of $\K_\Lambda$. Then for every object $C \in \K_\Lambda$ a right $\X$-approximation of $C$ must be a deflation. Indeed,  Let $X \xrightarrow{c} C$ be an $\X$-approximation of $C$. Since $\X$ is resolving, there exists a conflation \begin{tikzcd}[cramped, sep=small] 
		Y \arrow[r, "f", tail] & X' \arrow[r, "g", twoheadrightarrow] & C 
	\end{tikzcd} such that $Y \in \K$ and $X' \in \X$. Thus, there exists a map $X' \xrightarrow{x} X$  such that $g = c \cdot x$. By the octahedral axiom, we have a triangle $\Cone(x) \rightarrow Y[1] \rightarrow \Cone(c)$ such that the following diagram commutes 
\begin{center}
	\begin{tikzcd}[cramped]
		X' \arrow[r, "x"] \arrow[d, equal ] & X \arrow[r] \arrow[d, "c"] & \Cone(x) \arrow[d] \\
		X' \arrow[r, "g"] & C \arrow[r] \arrow[d] & Y[1] \arrow[d]\\
		& \Cone(c) \arrow[r, equal] & \Cone(c)
	\end{tikzcd}
\end{center}

But since $\Cone(x) \in \K^{[-2,0]}(\proj \Lambda)$ and $Y[1] \in \K^{[-2,-1]}(\proj \Lambda)$, $\Cone(c)$ must be in $\K^{[-2,0]}(\proj \Lambda) \cap \K^{[-3,-1]}(\proj \Lambda) = \K^{[-2,-1]}(\proj \Lambda)$. Then the triangle $\Cone(c)[-1] \rightarrow X \xrightarrow{c} C$ lies ins $\K_\Lambda$ and $c$ is a deflation. Since $\K_\Lambda$ satisfies WIC, is Krull-Schmidt and $\Hom$-finite, we can also find an approximation that is minimal. 
\end{remark}

\begin{proposition}\cite[Dual of Lemma 3.1]{liu2020hereditary}\label{approx-orth}
	Assume that $\K$ is Krull-Schmidt, Hom-finite and satisfies WIC. Let $\C \subset \K$ be an extension-closed subcategory of $\K$. If we have a conflation $X \rightarrowtail C \twoheadrightarrow Y$ where the corresponding deflation is a minimal right $\C$-approximation of $Y$, then $X \in \C^{\perp_1}$. 
\end{proposition}

\begin{proposition}\cite[Proposition 5.15]{adachi2022hereditary}\label{cotorres}
	Let $\K$ be a Krull-Schmidt, Hom-finite extriangulated category  satisfying WIC and having enough projectives and injectives. If $(\X, \Y)$ is a hereditary complete cotorsion pair, then $\X$ is a contravariantly finite resolving subcategory of $\K$. Reciprocally, if $\X \in \fres \K$, then $(\X, \X^{\perp_1})$ is a complete cotorsion pair.
\end{proposition}

Recall that there is a one-to-one correspondence between the set $\silt \K_\Lambda$ of isomorphism classes of basic silting objects in $\K_\Lambda$ and the set $\tautilt \Lambda$ of support $\tau$-tilting basic modules in $\Mod$ $\Lambda$ given by the map $H^0 : \K_\Lambda \rightarrow \Mod \Lambda$ (\cref{silt-tilt}). Moreover, the map  $M \mapsto \vartheta(M) = (\Fac(M), M^\perp)$ gives a correspondence between the sets $\tautilt \Mod \Lambda$ and $\ftor \Lambda$ (\cref{theo-mod}). The following result shows that these bijections are compatible to the ones described in \cref{cotorsiontorsion} and \cref{cotorsionsilting}.
	
\begin{proposition}\label{cotorsion_torsion} Let $\Lambda$ be a finite-dimensional $\Bbbk$-algebra and consider the bijections $\Phi : \cctor \K_\Lambda \rightarrow \ftor \Lambda$ of \cref{cotorsiontorsion}, as well as $\Psi : \cctor \K_\Lambda \rightarrow \silt \K_\Lambda$ of \cref{cotorsionsilting}. The following diagram
\begin{center}
	\begin{tikzcd}[row sep=2ex, column sep=large]
			& \cctor \K \arrow[r, "\Psi"]  &\silt \K & \\ 
			\mathclap{\K^{[-1,0]}(\proj \Lambda)} & & & \\
			{} \arrow[rrr, dashed, no head]& & & {} \\
			& & & \mathclap{\Mod \Lambda} \\
			& \ftor \Lambda \arrow[from = uuuu, "\Phi", crossing over, crossing over clearance=7ex] & \tautilt \Lambda \arrow[l, "\vartheta"] \arrow[from = uuuu, crossing over, crossing over clearance=7ex,"H^0"] &
	\end{tikzcd}
\end{center}
commutes.
\end{proposition}
\begin{proof}
	Let $(\X, \Y)$ be a complete cotorsion pair in $\K_\Lambda$. By \cref{cotorsionsilting}, since $\X$ is contravariantly finite and resolving, the complex $\Lambda[1] =$ \begin{tikzcd}[cramped, sep=small] 
		\Lambda \arrow[d] \\0\end{tikzcd} admits a conflation 
	\begin{equation}\label{lambda_res}
	\begin{tikzcd}[cramped] U_\Y \arrow[r, "u", tail] & U_\X \arrow[r, "\pi_\X", two heads] & \Lambda[1]
	\end{tikzcd}
	\end{equation}	 
where the corresponding deflation is a minimal right $\X$-approximation and $U_\Y \in \Y$ by \cref{approx-orth}. Since $\Y$ is closed by extensions and $\Lambda[1] \in \text{inj } \K_\Lambda \subset \Y$, we get that $U_\X \in \X\cap \Y$. Moreover, since the sequence \begin{tikzcd}[cramped, sep=small] \Lambda \arrow[r, "i_\Y", tail] &U_\Y \arrow[r, "u", two heads] & U_\X  
\end{tikzcd} is also a conflation, $\X$ is closed under extensions and $\Lambda \in \proj \K_\Lambda \subset \X$, then $U_\Y \in \X \cap \Y$. This implies that $\add(U_\X \oplus U_\Y) \subset \X \cap \Y$, and since $\Lambda \in \thi(\add(U_\X \oplus U_\X))$, we obtain that $\thi(\add(U_\X \oplus U_\X)) = \K_\Lambda$. By \cref{uincv}, we have that $\X\cap \Y = \add(U_\X \oplus U_\Y)$, which gives $\Psi((\X, \Y)) = U_{\X\cap\Y}$, where  $U_{\X\cap\Y}$ is the basic object such that $\add( U_{\X\cap\Y}) = \add(U_\X \oplus U_\X)$.
	
Let $\T = H^0(\Y)$ be the torsion class associated to $\Phi((\X, \Y))$. Applying $H^*$ to the conflation~(\ref{lambda_res}), we get the exact sequence $H^0(U_\Y) \rightarrow H^0(U_\X) \rightarrow 0$. Since $\T$ is closed under quotients, then $\Fac(H^0(U_\X \oplus U_\Y)) \subset \T$. On the other hand, by \cref{cotorsionsilting}, we know that $\Y = \add(U_\X \oplus U_\Y)^{\wedge}$, in particular, $\forall \ Y \in \Y$ there exists a conflation $Y' \rightarrowtail U \twoheadrightarrow Y$ where $Y' \in \Y$ and $U \in \add(U_\X \oplus U_\Y)$. Applying again the functor $H^*$, we get the exact sequence 
	\[ H^0(Y') \rightarrow H^0(U) \rightarrow H^0(Y) \rightarrow H^1(Y') = 0\]
which implies that $H^0(Y) \in \Fac H^0(U_\X \oplus U_\Y)$ for all $Y \in \Y$. Then $\T = \Fac H^0(U_\X \oplus U_\Y) = \Fac H^0(U_{\X\cap\Y})$, which gives the result. 
\end{proof}

\begin{corollary}\label{approxcotor}
For any complete cotorsion pair $(\X, \Y)$ in $K^{[-1,0]}(\proj \Lambda)$, there exists conflations
\begin{center}
	\begin{tikzcd}[cramped] U_\Y \arrow[r, "u", tail] & U_\X \arrow[r, "\pi_\X", two heads] & \Lambda[1]
	\end{tikzcd} \end{center}
\begin{center}	
	\begin{tikzcd}[cramped] \Lambda \arrow[r, "i_\Y"] & U_\Y \arrow[r, "u", tail] & U_\X
	\end{tikzcd}
\end{center}
where 
\begin{enumerate}[(i)]
	\item $U_\X \in \X$ and $U_\Y \in \Y $;
	\item $U_\X \oplus U_\Y$ is a silting object such that $\X \cap \Y = \add(U_\X \oplus U_\Y)$;
	\item $\pi_\X$ is a minimal right $\X$-approximation of $\Lambda[1]$;
	\item $i_\Y$ is a minimal left $\Y$-approximation of $\Lambda$.
\end{enumerate}

\end{corollary}

\begin{remark}\label{approximations}
	When $\X = (\add U)^\vee$ with $U \in \silt \K_\Lambda$, then $U_1 \twoheadrightarrow \Lambda[1]$  is a minimal $U$-right  approximation if and only if it is a minimal $\X$-right approximation. Indeed, by the proof of \cref{cotorsion_torsion}, we know that $\add(U_\X \oplus U_\Y) = \X \cap \Y = \add U$, which implies that $\pi_\X: U_\X \twoheadrightarrow \Lambda[1]$ is a minimal $U$-right approximation since $\add U \subset \X$. Consider now $\pi : U_1 \twoheadrightarrow \Lambda[1]$ a minimal $U$-right approximation. Since $U_1 \in \add U \subset \X$, there exists a map $f : U_1 \rightarrow U_\X$ such that $\pi = \pi_\X \circ f$. But $U_\X \in \add U$, so there is $g : U_\X \rightarrow U_1$ such that $\pi_\X = \pi \circ g$. Since $\pi = \pi \circ (g \circ f)$ and $\pi$ is minimal, $ g\circ f $ must be an isomorphism. Using that $\pi_\X$ is minimal as well, $f \circ g$ is also an isomorphism such that $\pi_\X = \pi_\X \circ (f\circ g)$. We conclude that $U_\X$ and $U_1$ are isomorphic. 
\end{remark}

\section{Main results}
\subsection{Thick subcategories and cotorsion pairs} The goal of this section is to prove the following result: 
 
\begin{theorem}\label{cotorsion-thick} Let $\Lambda$ be a finite-dimensional $\Bbbk$-algebra and let $\K_\Lambda =\K^{[-1,0]}(\proj \Lambda)$. There exist maps
	\begin{center}
		\begin{tikzcd}
			\ctor \K_\Lambda \arrow[r, shift left=.75ex, "\beta"] & \thi \K_\Lambda \arrow[l, shift left=.75ex, "\iota"]
		\end{tikzcd}
	\end{center}
	such that when restricted to complete cotorsion pairs and thick subcategories with enough injectives, they are inverse of each other.
\end{theorem}
\begin{proof} It follows from Proposition~\ref{cotorres}, \cref{beta}, \cref{betainj}, and \cref{betasurj}. 
\end{proof}

\begin{proposition}\label{beta}
	Let $\Lambda$ be a finite-dimensional $\Bbbk$-algebra. There exists a well defined map
	$$\res \K_\Lambda  \overset{\beta}{\longrightarrow} \thi \K_\Lambda $$
	which takes any $\X \in \res \K_\Lambda$ and sends it to $$\beta(\X) = \{ X \in \X \ | \ \forall \ \text{conflation } X \rightarrowtail X' \twoheadrightarrow X'' \text{ such that } X' \in \X, \text{ then } X'' \in \X\}.$$
\end{proposition}
\begin{proof}
	 Let $\X$ be a resolving subcategory of $\K_\Lambda$. First, we prove that $\beta(\X)$ is closed under direct summands. Suppose $X = X'\oplus X'' \in \beta(\X) \subset \X$, then $X'$ and $X''$ are in $\X$ since $\X$ is closed under direct summands. Let \begin{tikzcd}[cramped, sep=small] X' \arrow[r, "a", tail] & A \arrow[r, two heads] & B 
	\end{tikzcd} be a conflation with $A \in \X$, then \begin{tikzcd}[cramped, sep=small, ampersand replacement=\&] X'\oplus X'' \arrow[r, "{\left(\begin{smallmatrix} a & 0 \\ 0 & \Id_{X''} \end{smallmatrix}\right)}", tail] \&[6ex] A \oplus X'' \arrow[r, two heads] \& B 
	\end{tikzcd} is also a conflation with $A\oplus X'' \in \X$, which implies that $B \in \X$ since $X \in \beta(\X)$. Thus, $X' \in \beta(\X)$. 
	\smallskip
	
	Next, we prove that $\beta(\X)$ is closed under extensions. Consider a conflation $X \rightarrowtail X' \twoheadrightarrow X''$ in $\K_\Lambda$ such that $X, X'' \in \beta(\X)$. Since $\X$ is closed under extensions, $X'$ is in $\X$. Take $X' \rightarrowtail A \twoheadrightarrow B$ with $A \in \X$. By the octahedral axiom, there exists $C \in \K_\Lambda$ such that the diagram 
	\begin{center}
		\begin{tikzcd}
			X \arrow[r, tail] \arrow[d, equal] & X' \arrow[r, two heads] \arrow[d, tail] & X'' \arrow[d, tail]\\
			X \arrow[r, tail] & A \arrow[r, two heads] \arrow[d, two heads] & C \arrow[d, two heads] \\
			& B \arrow[r, equal] &B
		\end{tikzcd}
	\end{center}
	commutes and such that the last column and the middle row are conflations. Since $X \in \beta(\X)$, $C$ must be in $\X$. But $X'' \in \beta(\X)$ as well, so $B \in \X$. This implies that $X' \in \beta(\X)$ and $\beta(\X)$ is closed under extensions. 
	\smallskip
	
	We now prove that $\beta(\X)$ is closed under cones. Let $X \rightarrowtail X' \twoheadrightarrow X''$ be a conflation with $X, X' \in \beta(\X)$. In particular, $X' \in \X$, so $X'' \in \X$ by definition of $\beta(\X)$. Consider a conflation $X'' \rightarrowtail A \twoheadrightarrow B$ with $A \in \X$. Since $\Hom_{\mathcal{D}^b(\Mod \Lambda)}(B[-1], X[1]) = 0$, there exists $h : B[-1] \rightarrow X'$ and $C \in \K_\Lambda$, such that
	\begin{center}
		\begin{tikzcd}
			& B[-1] \arrow[r, equal] \arrow[d, "h"] &B[-1] \arrow[d] \arrow[dr, "0"] & \\
			X \arrow[r, tail] \arrow[d, equal] & X' \arrow[r, two heads] \arrow[d, tail] & X'' \arrow[d, tail] \arrow[r] & X[1]\\
			X \arrow[r, tail] & C \arrow[r, two heads] \arrow[d, two heads] & A \arrow[d, two heads] &\\
			& B \arrow[r, equal] &B &
		\end{tikzcd}
	\end{center}
is a commutative diagram where the second row and column are conflations. Since $\X$ is closed under extensions and $X, A \in \X$, we have that $C \in \X$. Likewise, $B$ must be in $\X$, since $X' \in \beta(\X)$, proving that $\beta(\X)$ is closed under cones. 
\smallskip
In order to prove that $\beta(\X)$ is closed under cocones, take now $X \rightarrowtail X' \twoheadrightarrow X''$ a conflation such that $X', X'' \in \beta(\X)$. Since $\X$ is resolving, it is closed under cocones and thus, $X \in \X$. Take $X \rightarrowtail A \twoheadrightarrow B$ a conflation in $\K_\Lambda$ with $A \in \X$. Using the octahedral axiom, we get the commutative diagram 
\begin{center}
	\begin{tikzcd}
		X''[-1] \arrow[r] \arrow[d, equal] & X \arrow[r, tail] \arrow[d, tail] & X' \arrow[r, two heads] \arrow[d, tail] & X'' \arrow[d, equal]\\
		X''[-1] \arrow[r] & A \arrow[r, tail] \arrow[d, two heads] & C \arrow[d, two heads] \arrow[r, two heads] & X''\\
		& B \arrow[r, equal] &B &
	\end{tikzcd}
\end{center}
Since both $\K$ and $\X$ are closed by extensions, $C \in \X$. Using that $X' \in \beta(\X)$, we get that $B \in \X$, which gives that $X \in \beta(\X)$, so it is closed under cocones.
\end{proof}

\begin{proposition}\label{iota}
	Let $\C \subset \K_\Lambda$ be an extension-closed subcategory of $\K_\Lambda$ that contains the zero object. Then $$\iota(\C) = \{X \in \K_\Lambda \ | \ \exists \ \text{an inflation }  X \rightarrowtail C \text{ with } C \in \C \}$$ is a resolving subcategory. Moreover, if $\C' \subset \C$ is also closed under extensions, then $\iota(\C') \subset \iota(\C)$.
\end{proposition}
\begin{proof}
Let $\C$ be a extension-closed subcategory of $\K_\Lambda$ containing the zero object. Note that for $P \in \proj \Lambda$, $P \rightarrow 0 \rightarrow P[1]$ is always a conflation. Since $0 \in \C$, we have that $\proj \Lambda \subset \iota(\C)$. That $\iota(\C)$ is closed under cocones and direct summands follows directly from the definition of $\iota$. Take $X \rightarrowtail X' \twoheadrightarrow X''$ a conflation where $X, X'' \in \iota(\C)$. In particular, there is a conflation $X'' \rightarrowtail C \twoheadrightarrow W$ where $C \in \C$. Using that $\Hom_{\mathcal{D}^b(\Mod \Lambda)}(W[-1], M[1]) = 0$ and the octahedral axiom, we get the commutative diagram 
	\begin{center}
	\begin{tikzcd}
		& W[-1] \arrow[r, equal] \arrow[d] &W[-1] \arrow[d] \arrow[dr, "0"] & \\
		X \arrow[r, tail] \arrow[d, equal] & X' \arrow[r, two heads] \arrow[d, tail] & X'' \arrow[d, tail] \arrow[r] & X[1]\\
		X \arrow[r, tail] & A \arrow[r, two heads] \arrow[d, two heads] & C \arrow[d, two heads] &\\
		& W \arrow[r, equal] &W &
	\end{tikzcd}.
\end{center}
But $X \in \iota(\C)$ as well, so there exists a conflation $X \rightarrowtail C' \twoheadrightarrow W'$ with $C' \in \C$. Using the octahedral axiom once more, we can construct the commutative diagram 
\begin{center}
	\begin{tikzcd}
		C[-1] \arrow[r] \arrow[d, equal] & X \arrow[r, tail] \arrow[d, tail] & A \arrow[r, two heads] \arrow[d, tail] & C \arrow[d, equal]\\
		C[-1] \arrow[r] & C' \arrow[r, tail] \arrow[d, two heads] & B \arrow[d, two heads] \arrow[r, two heads] & C\\
		& W' \arrow[r, equal] &W' &
	\end{tikzcd}
\end{center}
and since $\C$ is closed under extensions, $B \in \C$.  Composing the inflations $X' \rightarrowtail A \rightarrowtail B$, we get that $X' \in \iota(\C)$.
\end{proof}

\begin{remark}
	Both the map $\beta$ defined in \cref{beta} as well as the map $\iota$ in \cref{iota} can be thought as dual to the maps first proposed by C.~Ingalls and H.~Thomas in \cite{ingalls2009noncrossing} between wide subcategories and torsion classes. The definition of $\beta$ is inspired by $\alpha$ in \cref{theo-mod}. In \cite{ingalls2009noncrossing}, the map taking a wide subcategory $\W$ of an hereditary algebra to a torsion class was defined as $\Fac(\W)$. In other for this map to be defined for any algebra, F.~Marks and J.~\v{S}t'ov\'\i\v{c}ek modified it to $\Filt(\Fac(\W))$ in \cite{marks2017torsion}. Our definition of $\iota$ recalls C.~Ingalls and H.~Thomas original map, and the arguments used in \cref{iota} to prove that $\iota$ is well defined rely on the fact $\K_\Lambda$ is an hereditary extriangulated category. 
\end{remark}

\begin{lemma}\label{lemmaiota}
	Let $\X$ be a contravariantly resolving subcategory of $\K_\Lambda$. Then $~\iota(\X) \subset \X$.
\end{lemma}
\begin{proof}
	Let $\X$ be a contravariantly finite resolving subcategory of $\K_\Lambda$ and let $X \in \iota(\X)$. Take a conflation $X \rightarrowtail T \twoheadrightarrow Y$ such that $T \in \X$. Since $\X$ is contravariantly finite, by \cref{approx-orth} there exists $X' \in \X$, $Y' \in \X^{\perp_1}$ and a conflation $Y' \rightarrowtail X' \twoheadrightarrow X$ such that the corresponding deflation $X' \twoheadrightarrow X$ is a minimal $\X$-right approximation. By the octahedral axiom, there exists $C \in \K_\Lambda$ such that the following diagram is commutative
	\begin{center}
		\begin{tikzcd}
			Y[-1] \arrow[r] \arrow[d, equal] & X' \arrow[r, tail] \arrow[d, two heads] & C \arrow[r, two heads] \arrow[d] & Y \arrow[d, equal]\\
			Y[-1] \arrow[r] & X \arrow[r, tail] \arrow[d] & T \arrow[d] \arrow[r, two heads] & Y\\
			& Y'[1] \arrow[r, equal] &Y'[1] &
		\end{tikzcd}
	\end{center}
	But $\EE(T, Y') = 0$, since $T \in \X$ and $Y' \in \X^{\perp_1}$. In particular, $C \simeq T \oplus Y'$. This implies that $X' \simeq X \oplus Y'$, and therefore $X \in \X$ because $\X$ is closed under direct summands. 
\end{proof}
\begin{lemma}\label{betainj}
	Let $\X$ be a contravariantly finite resolving category of $\K_\Lambda$, then $$\iota(\beta(\X)) = \X.$$ 
\end{lemma}
\begin{proof}
Since $\beta(\X) \subset \X$, the previous lemma shows that $\iota(\beta(\X)) \subset \iota(\X) \subset \X$.
Consider now $U_\X$ as in \cref{approxcotor}. We will show that $U_\X \in \beta(\X)$. Let \begin{tikzcd}[cramped, sep=small] U_\X \arrow[r, "x", tail] & X \arrow[r, two heads] & X'
\end{tikzcd} be a conflation with $X \in \X$. By the octahedral axiom, there exists $W \in \K_\Lambda$ and a commutative diagram
	\begin{center}
	\begin{tikzcd}
		\Lambda \arrow[r] \arrow[d, equal] & U_\Y \arrow[r, tail] \arrow[d, tail] & U_\X \arrow[r, two heads, "\pi_\X"] \arrow[d, "x"] & \Lambda[1]\arrow[d, equal]\\
		\Lambda \arrow[r] & W \arrow[r, tail] \arrow[d] & X \arrow[d] \arrow[r, two heads, "\pi'_\X"] & \Lambda[1]\\
		& X' \arrow[r, equal] &X' &
	\end{tikzcd}
\end{center}
such that the second line is a conflation. Since $\pi_\X$ is a minimal $\X$-approximation, there exists $x' : X \rightarrow U_\X$ such that $\pi'_\X = \pi_\X \circ x'$, which implies that $\pi_\X \circ (x' \circ x) = \pi'_\X \circ x = \pi_\X$. Since $\pi_\X$ is minimal, we get that $x'\circ x$ is an isomorphism. In particular, $x$ is a section, which implies that $X'$ is a direct summand of $X\in \X$. This gives that $X' \in \X$ and $U_\X \in \beta(\X)$.

Since we have an inflation $U_\Y \rightarrowtail U_\X$, $U_\Y\in \iota(\beta(\X))$ and $\add(U_\X \oplus U_\Y) \subset \iota(\beta(\X))$. But $\iota(\beta(\X))$ is closed under cocones, so by \cref{vee-wedge-desc}, $\X = \add(U_\X \oplus U_\Y)^{\vee} \subset \iota(\beta(\X))$. 
\end{proof}

\cref{betainj} tell us that, when restricted to contravariantly finite resolving categories, the map $\beta$ is injective. The following results will allow us to explicitly describe the image of $\beta$. 

\begin{proposition}\label{ccone}
	Let $(\X, \Y)$ be a complete cotorsion pair in $\K_\Lambda$, then $$\X = \CCone(\X\cap\Y, \X\cap\Y).$$
\end{proposition}

\begin{proof}
	Recall that $\mathcal{U} = \X\cap\Y = \add (U_\X\oplus U_\Y)$ and that $\X = \mathcal{U}^{\vee}$ by \cref{approxcotor}. Let $X \in \X$, there must exist $m \in \mathbb{Z}_{\geq 0}$ such that $X \in \mathcal{U}^{\vee}_m$, that is, we can find conflations
	\begin{equation}\label{T1}
		X \rightarrowtail U_0 \twoheadrightarrow X_1 \dashrightarrow X[1] 
	\end{equation}
	\begin{equation}\label{T2}
		X_1 \rightarrowtail U_1 \twoheadrightarrow X_2 \dashrightarrow X_1[1]
	\end{equation}
	with $X_i \in \mathcal{U}^{\vee}_{m-i} \subset \X$ for $i = 1,2$, and $U_0, U_1 \in \mathcal{U}$. Shifting and rotating triangles~(\ref{T1}) and (\ref{T2}), and using that $\Hom_{\mathcal{D}^b(\Mod \Lambda)}(X_2, X[2]) = 0$, as well as the octahedral axiom, we get a commutative diagram
	\begin{center}
		\begin{tikzcd}
			X_2 \arrow[r, equal] \arrow[d] & X_2 \arrow[d, "0"]& \\
			X_1[1] \arrow[r] \arrow[d] &X[2] \arrow[r] \arrow[d] &U_0[2] \arrow[d, equal] \\
			U_1[1] \arrow[r] & X_2[1] \oplus X[2] \arrow[r] & U_0[2]
		\end{tikzcd}
	\end{center}
	where the last row is a triangle. Then $U_0 \rightarrow U_1 \rightarrow X_2 \oplus X[1] \dashrightarrow U_0[1]$ is a triangle as well. Since $U_0 \in \Y$, $\EE(X_2, U_0) = 0$, the morphism $X_2 \oplus X[1] \dashrightarrow U_0[1]$ must be of the form $X_2 \oplus X[1] \xrightarrow{(0, f)} U_0[1]$. This in turn implies that $U_1 \simeq X_2 \oplus \Cone(f)[-1]$, thus $U' = \Cone(f[-1]) = \Cone(f[-1])$ belongs to $\mathcal{U}$ since $\U$ is closed under direct summands. Remark that, by the commutativity of the previous diagram, $f[-1]$ is exactly the inflation $X \rightarrowtail U_0$. We get that $U' \simeq X_0$, and so, $X \in \CCone(\mathcal{U}, \mathcal{U})$.
\end{proof}

\begin{lemma}\label{betaY}
	Let $(\X, \Y)$ be a complete cotorsion pair in $\K_\Lambda$ and consider the conflation $U_\Y\rightarrowtail U_\X \twoheadrightarrow \Lambda[1]$ as in \cref{approxcotor}. Then $$\beta(\X) \cap \Y = \add (U_\X).$$
\end{lemma}
\begin{proof}
	By the proof of \cref{betainj}, we know that $U_\X \in \beta(\X) \cap \Y$. Since both $\beta(\X)$ and $\Y$ are additive subcategories, we get that $\add U_\X \subset \beta(\X) \cap \Y$. Now take $Y \in \beta(\X) \cap \Y \subset \X \cap \Y= \add(U_\X \oplus U_\Y)$ and suppose that $Y$ is indecomposable.  Recall that $U_\X$ and $U_\Y$ share no non-zero direct summands \cite[Lemma 2.25]{aihara2012silting}, so we can also suppose that $Y$ is a direct summand of $U_\Y = Y \oplus Y'$. Using the octahedral axiom, we have the commutative diagram 
	\begin{center}
		\begin{tikzcd}
			Y \arrow[r, equal] \arrow[d, tail]&Y \arrow[d, tail]&\\
			U_\Y \arrow[r, tail] \arrow[d, two heads] &U_\X \arrow[r, two heads] \arrow[d, two heads] &\Lambda[1] \arrow[d, equal]\\
			Y' \arrow[r, tail] & C \arrow[r, two heads] & \Lambda[1]
		\end{tikzcd}
	\end{center}
	Since $Y \in \beta(\X)$, the complex $C$ must be in $\X$, which implies that $\EE(C, Y) = 0$. That is, $U_\X \simeq Y \oplus C$ and $Y \in \add(U_\X) \cap \add(U_\Y) = \{0\}$. Thus $\beta(\X) \cap \Y = \add(U_\X)$.
\end{proof}

\begin{lemma}\label{desc_beta}
	Let $\X$ be a contravariantly finite resolving subcategory of $\K_\Lambda$ and let $U_\X$ be as in \cref{approxcotor}. Then $$\beta(\X) = \CCone(\add(U_\X), \add(U_\X)) = \thi(U_\X).$$
	In particular, $\beta(\X)$ is a thick subcategory with enough injectives. All the injectives objects of $\beta(\X)$ lie in $\add U_\X$ and all objects in $\beta(\X)$ have injective dimension $\leq 1$.
\end{lemma}

\begin{proof}
	Let $\U_\X = \add(U_\X)$ and take $U, U' \in \U_\X \subset \Y$. For every conflation $U \rightarrowtail U' \twoheadrightarrow U''$, we have that $U'' \in \Y$ since $\Y$ is closed under cones. Moreover, $U, U' \in \beta(\X)$ which is thick, so $U'' \in \beta(\X) \cap \Y$, the later being equal to $\U_\X$ by \cref{betaY}. Thus, $\U_\X$ is a presilting subcategory that is closed under cones. By \cref{vee-wedge-desc} we get that $$\thi (U_\X) = \U_\X^\vee.$$
	On the other hand, \cref{ccone} tells us that $\X = \CCone(\U, \U)$. So for every $X \in \beta (\X) \subset \X$ there exists a conflation 
	\[X \rightarrowtail U_0 \twoheadrightarrow U_1\]
	where $U_i \in \U$ for $i = 0,1$. We know that there exists $U^\X_0 \in \U_\X$ and $U_0^\Y \in \U_\Y$ such that $U_0 \simeq U_0^\X \oplus U_0^\Y$. Since $U_0^\Y$ is in $\add(\U_\Y)$, there exists $V \in \add(\U_\Y)$  and $m \in \mathbb{Z}_{\geq 0}$ such that $U_0^\Y \oplus V \simeq U_\Y^{\oplus m}$. We get a conflation $X \rightarrowtail U_0^\X \oplus U_\Y^{\oplus m} \twoheadrightarrow U_1 \oplus V$. Applying the octahedral axiom, we get the commutative diagram
	\begin{center}
		\begin{tikzcd}[ampersand replacement=\&, row sep=6ex]
			X \arrow[r] \arrow[d, equal] \& U_0^\X\oplus U_\Y^{\oplus m} \arrow[r] \arrow[d, tail, "{\left(\begin{smallmatrix} \Id_{U^\X_0} & 0 \\ 0 & u^{\oplus m} \end{smallmatrix}\right)}"] \& U_1 \oplus V \arrow[d, tail]\\
			X \arrow[r, tail] \& U_0^\X \oplus U_\X^{\oplus m} \arrow[r, two heads] \arrow[d, two heads] \& C \arrow[d, two heads] \\
			\& \Lambda[1]^{\oplus m} \arrow[r, equal] \&\Lambda[1]^{\oplus m}
		\end{tikzcd}
	\end{center}
	Since $X, U_0^\X \oplus U_\X^{\oplus m} \in \beta(X)$, the complex $C$ must lie in $\beta(\X)$, since it is a thick subcategory of $\K_\Lambda$. Moreover, $C \in \Y$, because $\Y$ is closed under extensions and contains $U, V$ and $\Lambda[1]$. This implies that $C \in \beta(\X) \cap \Y = \U_\X$. In particular, $X \in \CCone(\U_\X, \U_\X)$. We conclude that $\beta(\X) \subset \CCone(\U_\X, \U_\X) \subset \U_\X^\vee = \thi(U_\X)$. Since $\thi(U_\X)$ is the smallest thick subcategory containing $U_\X$, $\beta(\X) =  \CCone(\U_\X, \U_\X) = \thi(U_\X)$.
	\smallskip
	We know show that any $U \in \U_\X$ is an injective object in $\beta(\X)$. Consider a conflation $U \rightarrowtail Y \twoheadrightarrow X$ with $Y,X \in \beta(\X)$, then there must exists a conflation $X \rightarrowtail U' \twoheadrightarrow U''$ with $U', U'' \in \U_\X$. We can find $A \in \K_\Lambda$ such that the follow diagram 
	\begin{center}
		\begin{tikzcd}
			U \arrow[r, tail, "f"] \arrow[d, equal] & Y \arrow[r, two heads] \arrow[d, tail, "g"] & X \arrow[d, tail] \\
			U \arrow[r, tail, "f'"] & A \arrow[r, two heads] \arrow[d, two heads] & U' \arrow[d, two heads]\\
			& U'' \arrow[r, equal] &U'' 
		\end{tikzcd}.
	\end{center}
commutes. Since $\EE(U,U') =0$, the second line splits and $A \in \U_\X$. That is, there exists $h : A \rightarrow U$ such that  $h \circ f' = \Id_U$, which in turn implies that $(h\circ g) \circ f = h\circ (f\circ g) = h \circ f' = \Id_U$. We conclude that $f$ is a section, so $U \rightarrowtail Y \twoheadrightarrow X$ splits, and $U$ must be injective. That all injective objects are in $\U_\X$ follows directly from the fact that $\beta(\X) = \CCone(\U_\X, \U_X)$. This finishes the proof.
\end{proof}

\begin{proposition}\label{lfinitethi}
	Let $\T \subset \K$ be a thick subcategory of an hereditary extriangulated category. Then, $\T$ has enough injectives if and only if there exist a presilting subcategory $\U \subset \K$ such that $\U$ is closed under cones and $\T = \thi(\U)$.
\end{proposition}
\begin{proof}
	Suppose $\T$ has enough injectives and let $\U = \inj \K$, then $\T = \U^\vee$. Since $\EE(\T, \U) = 0$, in particular we have that $\EE(\U, \U) = 0$, so $\U$ is presilting. For any conflation $U \rightarrowtail U' \rightarrow X$ with $U, U' \in \U$, we must have that $X \in \T$ since $\T$ is thick. Moreover, $U$ is injective, so the conflation must split and $X \in \U$, which in turn implies that $\U$ is closed under cones. We conclude that $\T = \U^\vee = \thi(\U)$. 
	
	Conversely, suppose that $\T = \thi(\U)$, where $\U$ is presilting and closed under cones. To prove the result, it suffices to show that every $U \in \U$ is injective. Indeed, since $\U$ is closed under cones, we have that $\U^\vee = \thi(\U) = \T$, so any object in $\T$ can be approximated by objects in $\U$. Let $U \in \U$ and take a conflation $U \rightarrowtail X \twoheadrightarrow Y$. Since $\T = \U^\vee$, there exist a conflation $X \rightarrowtail U' \twoheadrightarrow X'$ with $U' \in \U$ and $X' \in \T$. Then, there exist $A \in \T$ and a commutative diagram 
	\begin{center}
		\begin{tikzcd}
			U \arrow[r, tail, "f"] \arrow[d, equal] & X \arrow[r, two heads] \arrow[d, tail, "g"] & Y \arrow[d, tail] \\
			U \arrow[r, tail, "f'"] & U' \arrow[r, two heads] \arrow[d, two heads] & A \arrow[d, two heads]\\
			& U'' \arrow[r, equal] &U'' 
		\end{tikzcd}.
	\end{center}
 where the second line is a conflation. But $\U$ is closed under cones, so $A \in \U$. Moreover, $\U$ is presilting, so the second line must split, which implies that $U \rightarrowtail X \twoheadrightarrow Y$ does as well. We conclude that $U$ is injective.
\end{proof}

\begin{lemma}\label{betasurj}
	Let $\T$ be a thick subcategory of $\K_\Lambda$ with enough inectives, then
	$$\beta(\iota(\T)) = \T.$$
\end{lemma}
\begin{proof}
	Let $\T$ a thick subcategory of $\K_\Lambda$ with enough inectives, by \cref{lfinitethi} we know that there exists a basic presilting object $U$ such that $\U = \add(U)$ is closed under cones and $\T = \thi (\U)$. Consider now its Bongartz completion \cite{aihara2012silting} $\overline{U} = U' \oplus V$ given by the conflation 
	\begin{equation}\label{bongcom}
		V_0 \rightarrowtail U_0 \twoheadrightarrow \Lambda[1]
	\end{equation}
	where the deflation $U_0 \twoheadrightarrow \Lambda[1]$ is a minimal $U$-right approximation of $\Lambda[1]$, $\add U_0 = \add U' \subset \add U$ and $\add V_0 =\add V$ with $V$ and $U'$ basic. Recall that $\overline{U}$ is silting and that $U$ is a direct summand of $\overline{U}$. By construction, $U_0 \twoheadrightarrow \Lambda[1]$ is also a minimal $\overline{U}$-right approximation and $\add U' \cap \add V = \{0\}$. Now let $W \in \add U / \add U'$, such that $W$ is indecomposable. Since $W \in \add \overline{U}= \add U' \sqcup \add V$, there exists $W' \in \add \overline{U}$ such that $V_0 = W \oplus W'$. We can find $A \in \K_\Lambda$ and a commutative diagram 
	\begin{center}
		\begin{tikzcd}
			W \arrow[r, tail] \arrow[d, equal] & V_0 \arrow[r, two heads] \arrow[d, tail] & W' \arrow[d, tail] \\
			W \arrow[r, tail] & U_0 \arrow[r, two heads] \arrow[d, two heads] & A \arrow[d, two heads]\\
			& \Lambda[1] \arrow[r, equal] & \Lambda[1] 
		\end{tikzcd}.
	\end{center}
	such that the second line is a conflation. But $W$ and $U_0$ lie in $\add U$ wich is closed under cones, so $A \in \add U$. Since $U$ is presilting, the second line must split, in particular $W \in \add U' \cap \add V = \{0\}$. We conclude that $\add U' = \add U$. 
	
	Now, let $\X = (\add \overline{U})^\vee$, by \cref{cotorsionsilting} and \cref{approximations}, we know that $(\X, \X^{\perp_1})$ is a cotorsion pair and that the deflation in the conflation~(\ref{bongcom}) is a minimal $\X$-right approximation of $\Lambda[1]$. Since $\add U_0 = \add U$, we have that 
	\begin{equation}\label{eq1}
		 \beta(\X) = \CCone(U,U) = \thi (U) = \T. 
	\end{equation}
	Finally, we know that $\iota(\T)$ is resolving, that is closed under extensions, direct summands and cocones. Since $V \in \iota(\T)$, then $\X = (\add \overline{U})^\vee \subset \iota(\T)$. Moreover, $\T = \CCone(U,U) \subset (\add \overline{U})^\vee$ and both subcategories are extension-closed, so \cref{iota} and \cref{lemmaiota} imply that $\iota(\T) \subset \iota((\add \overline{U})^\vee) = \iota(\X) \subset \X$. This implies that 
	\begin{equation}\label{eq2}
		\X = (\add \overline{U})^\vee = \iota(\T)
	\end{equation}
	Putting (\ref{eq1}) and (\ref{eq2}) together, we get that $$\beta(\iota(\T)) = \beta(\X) = \T$$
	which gives the result. 
	
\end{proof}
 
\subsection{Linking thick and wide subcategories}\label{stabinK}
The connections between $\tau$-tilting theory and stability conditions have been studied by a vast number of authors in the last two decades, resulting in a direct bridge between s$\tau$-tilting modules, torsion classes and semistable subcategories. In this section we propose a notion of semistability for objects in $\K_\Lambda = \K^{[-1,0]}(\proj \Lambda)$ that will allow us to construct a bridge between the bijections of \cref{theo-mod} and those of \cref{cotorsion-thick}.

\begin{definition}[\textbf{$M$-semistability}]\label{defMsemistab1}
Let $M \in\Mod\Lambda$ and $X =$ \begin{tikzcd}[cramped, sep=small] 
		X^{-1} \arrow[d, "x"] \\ X^0 \end{tikzcd} $\in \K^{[-1,0]}(\proj \Lambda)$. We say that $X$ is \textbf{$M$-semistable} if the map $\Hom(X^0, M) \xrightarrow{x^*} \Hom(X^{-1}, M)$ is an isomorphism of $\Bbbk$-vector spaces. In particular, since $\langle [X] , [M ]\rangle = \dim_\Bbbk(\Hom(X^0,M)) - \dim_\Bbbk(\Hom(X^{-1}, M))$, if $X$ is $M$-semistable, $\langle [X] , [M ]\rangle = 0$.
\end{definition}

Note that this definition does not depend on the choice of representative of $X$ in its isomorphism class inside $\K_\Lambda$ thanks to \cref{definf}. In \cref{s_geometry} we will discuss the geometric origin of this notion, but for now, let us proceed to the proof of the main theorem of this section. 

\begin{definition}
	Let $\mathcal{H}$ be a subcategory of $\Mod \Lambda$. We define $\Tt(\mathcal{H})$ to be the full subcategory of $\mathcal{K}^{[-1,0]} (\proj \Lambda)$ whose objects are all complexes $X$ such that $X$ is $N$-semistable $ \forall \ N \in \mathcal{H}$. Similarly, if $\mathcal{C} \subset \mathcal{K}^{[-1,0]} (\proj \Lambda)$, we define $\Ww(\mathcal{C})$ as the full subcategory of modules $N$ such that all objects in $\C$ are $N$-semistable.
\end{definition}

\begin{proposition}\label{wide-thick} Let $\Lambda$ be a finite-dimensional $\Bbbk$- algebra, then
	\begin{enumerate}
		\item $\forall \ \mathcal{H} \subset \Mod \Lambda$, $\Tt(\mathcal{H}) \subset \K^{[-1,0]} (\proj \Lambda)$ is a thick subcategory.
		\item $\forall \C \subset \K^{[-1,0]} (\proj \Lambda)$, $\Ww(\mathcal{C}) \subset \Mod \Lambda$ is a wide subcategory.
		\item For any subcategories $\mathcal{H} \subset \Mod \Lambda$ and $\mathcal{C} \subset \K^{[-1,0]}(\proj \Lambda)$,
		\begin{gather*}
			\mathcal{H} \subset \Ww(\Tt(\mathcal{H})),\\
			\mathcal{C} \subset \Tt(\Ww(\mathcal{C})).
		\end{gather*}
		\item For any subcategories $\mathcal{H} \subset \Mod \Lambda$ and $\mathcal{C} \subset \K^{[-1,0]}(\proj \Lambda)$.
		\begin{gather*}
			\Tt(\mathcal{H}) = \Tt(\wide(\mathcal{H})) \\
			\Ww(\C) = \Ww(\thi(\C))
		\end{gather*}
	\end{enumerate} 
\end{proposition}

\begin{proof}
	We only prove (1) for $\mathcal{H} = \{M\}$ with $M \in \Mod \Lambda$. The result follows noting that $\Tt(\mathcal{H}) = \bigcap_{M \in \mathcal{H}} \Tt(M)$. The statement in (2) follows using similar arguments. 
	
	\textit{Closure under extensions: } Let $X \rightarrowtail Y \twoheadrightarrow Z$ be a conflation and suppose $X$ and $Z$ are in $\Tt(M)$. As we have seen in \cref{definf}, we can find $P \in \proj  \Lambda$ such that $X \rightarrowtail Y \oplus \text{\begin{tikzcd}[cramped, sep=tiny]  P \arrow[d, equal]\\ P \end{tikzcd}} \twoheadrightarrow Z$ is a conflation in $\mathcal{C}^{[-1,0]}(\proj \Lambda)$. Applying $\Hom(-,M)$ to the associated exact sequences, we get the commutative diagram
	\begin{center}
		\begin{equation}\label{cdiag}
			\begin{tikzcd}[cramped, column sep=small]
				0 \arrow[r] & \Hom(Z^0, M) \arrow[r] \arrow[d, "z^*"] & \Hom(Y^0 \oplus P, M) \arrow[r] \arrow[d, "(y\oplus \Id_P)^*"] &\Hom(X^0, M) \arrow[r] \arrow[d, "x^*"] & 0 \\
				0 \arrow[r] & \Hom(Z^{-1}, M) \arrow[r]& \Hom(Y^{-1} \oplus P, M) \arrow[r] & \Hom(X^{-1}, M) \arrow[r] & 0
			\end{tikzcd}
		\end{equation}	
	\end{center}
	where $z^*$ and $x^*$ are isomorphisms by hypothesis. Since $(y\oplus \Id_P)^* = y^* \oplus \Id_{\Hom(P,M)}$, $y^*$ is an isomorphism as well.
	
	\textit{Closure under cones and cocones: }Take a conflation as before and suppose that $X$ and $Y$ are $M$-semistable. In particular, $0 = \langle [Y], [M] \rangle = \langle [X], [M] \rangle + \langle [Z] , [M] \rangle = 0 + \langle [Z], [M] \rangle$. As before, there exists $P \in \proj \Lambda$ and a commutative diagram like (\ref{cdiag}). Since $(y \oplus \Id_P)^*$ and $x^*$ are isomorphisms we can deduce that $z^*$ is injective. Moreover, $z^*$ is a linear map between vector spaces of same dimension, so it must be bijective. The proof of $\Tt(M)$ being closed by cocones is dual. 
	
	\textit{Closure under direct summands:} Let $ X \in \Tt(M)$ such that $X \simeq X' \oplus X''$. Since we have inflations $X' \rightarrowtail X$ and $X'' \rightarrowtail X$, then $\langle [X'], [M]\rangle, \langle [X''], [M]\rangle \geq 0$ (see \cref{Mimplies[M]} for more details). But $0 = \langle [X], [M] \rangle = \langle [X'], [M]\rangle + \langle[X''], [M]\rangle$, so both terms must be equal to $0$. Take now the conflation $X'' \rightarrowtail X \twoheadrightarrow X'$ and $P \in \proj \Lambda$ such that we have a commutative diagram like (\ref{cdiag}). This time around, $x^* \oplus \Id_P^*$ is an isomorphism, which implies that $(x')^*$ is bijective since it is injective and $\langle [X'], [M]\rangle = 0$.
	
	We proceed to prove (3). Since for any subcategory $\mathcal{H} \subset \Mod \Lambda$, $\Tt(\mathcal{H}) = \bigcap_{M \in \mathcal{H}} \Tt(M)$, the map $\Tt$ reverses inclusions. Take $M \in \mathcal{H}$, then all $X \in \Tt(\mathcal{H})$ satisfy that $X$ is $M$-semistable, that is, $M \in \Ww(\Tt(\mathcal{H}))$. The rest of the statement follows from similar arguments.
	
	Lastly, consider $\C \subset \K^{[-1,0]}(\proj \Lambda)$. Since $\Ww$ reverses inclusions and $\C \subset \thi(\C)$, then $\Ww(\thi(\C)) \subset \Ww(\C)$. We have seen that $\Tt(\Ww(\C))$ is a thick subcategory that contains $\C$, so $\thi(\C) \subset \Tt(\Ww(\C)) $. Applying (3) to $\mathcal{H} = \Ww(C)$, we have the following inclusions : 
	
		\[\Ww(\C) \subset \Ww(\Tt(\Ww(\C))) \subset \Ww (\thi(\C)) \subset \Ww(\C).\]
	Thus $\Ww(\C) = \Ww(\thi(\C))$. That $\Tt(\mathcal{H}) = \Tt(\wide(\mathcal{H}))$ for any subcategory $\mathcal{H}$ of $\Mod \Lambda$ follows from the same argument. This proves (4).
	
\end{proof}

For any extriangulated category $\K$, we will denote by $\fthi \K$ the set of all thick subcategories of $\K$ that have enough injectives. We are now ready to state the main theorem of this section. 

\begin{theorem}\label{theo-thick-wide}
	Let $\Lambda$ be a finite-dimensional $\Bbbk$-algebra and take $\K_\Lambda$ as before. There exist well defined maps
	\begin{center}
		\begin{tikzcd}
			\wide \Lambda \arrow[r, shift left=.75ex, "\Tt"] & \thi \K_\Lambda \arrow[l, shift left=.75ex, "\Ww"] 
		\end{tikzcd}
	\end{center}
	such that, when restricted to thick subcategories with enough injectives and lef finite wide subcategories, they make the following diagram commute
	\begin{center}
		\begin{tikzcd}[cramped,column sep=large, row sep=scriptsize]
			& & \silt \K_\Lambda \arrow[dr, "\thi(U_\rho)"] & & \\
			&  \cctor\K_\Lambda \arrow[ur, "\Psi"]  \arrow[rr, "\beta", crossing over]& & \fthi \K_\Lambda& \\ 
			\mathclap{\K^{[-1,0]}(\proj \Lambda)} & & & &  \\
			{} \arrow[rrrr, dashed, no head, crossing over]& & & & {} \\
			& & & & \mathclap{\Mod \Lambda} \\
			& \ftor\Lambda \arrow[from= uuuu, "\Phi", crossing over, crossing over clearance=7ex] & & \fwide\Lambda \arrow[from= ll, "\alpha"] \arrow[from= uuuu, crossing over, crossing over clearance=7ex,"\Ww"] & 
		\end{tikzcd}
	\end{center} 
	In particular, $\Ww$ and $U \in \silt \K_\Lambda \mapsto \thi(U_\rho) \in \fthi \K_\Lambda$ are bijective.
\end{theorem}

\begin{proof}
The first part of the statement follows from \cref{wide-thick}. We show that the center square of the diagram is commutative, that the upper triangle is as well follows from \cref{desc_beta}. Let $(\X, \Y)$ be a complete cotorsion pair. We prove that $\Ww(\beta(\X)) = \alpha(H^0(\Y))$. \cref{cotorsion_torsion} implies that $H^0(\Y) = \Fac(H^0(U_\X\oplus U_\Y))$. By \cite[Lemma 3.8]{marks2017torsion} and \cite[Lemma 3.5]{yurikusa2018wide} we have that $\alpha(\Fac(H^0(U_\X\oplus U_\Y))) = \Ww(U_\X)$. Moreover, $\Ww(\beta(\X)) = \Ww(\thi (U_\X))$ by \cref{desc_beta}, and $\Ww(U_\X) = \Ww(\thi(U_\X))$ by \cref{wide-thick} (4). Putting all these equalities together we get that 
	$$\Ww(\beta(\X)) = \Ww(\thi (U_\X)) = \Ww(U_\X) = \alpha(\Fac(H^0(U_\X\oplus U_\Y))) = \alpha(H^0(\Y))$$
	which gives us the result. 
\end{proof}
	
\begin{example}
Let $Q$ be the quiver $1 \xrightarrow{\alpha} 2 \xrightarrow{\beta} 3$. Then the Auslander-Reiten quiver of $\Mod \Bbbk Q$ is the following
\begin{center}
	\begin{tikzpicture}
		\node (P3) at (0,0) {$P_3 = I_1$};
		\node (P2) at (-1,-1) {$P_2$}; \node (I2) at (1,-1) {$I_2$};
		\node (P1) at (-2,-2) {$P_1$}; \node (S2) at (0,-2) {$S_2$}; \node (I3) at (2,-2) {$I_3$};
		\draw[->] (P1) --(P2); \draw[->] (P2) --(P3); \draw[->] (P3) --(I2); \draw[->] (I2) --(I3); \draw[->] (P2) --(S2); \draw[->] (S2) --(I2);
	\end{tikzpicture} .
\end{center}
All minimal projective presentations of indecomposable modules are indecomposable objects in $\K_{\Bbbk Q} = \K^{[-1,0]}(\proj \Bbbk Q)$, as are the objects $P \rightarrow 0$ where $P$ is an indecomposable projective module. Then, the AR quiver of $\K_{\Bbbk  Q}$ is given by
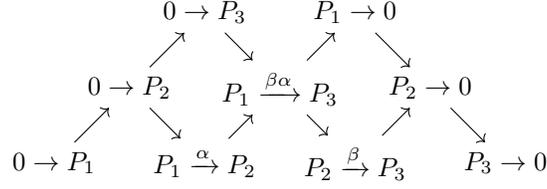
\begin{center}
	\begin{tikzpicture}
		\node (P3) at (0,0) {$0 \rightarrow P_3$}; \node (P1s) at (2,0) {$P_1 \rightarrow 0$};
		\node (P2) at (-1,-1) {$ 0 \rightarrow P_2$}; \node (I2) at (1,-1) {$P_1 \xrightarrow{\beta \alpha} P_3$}; \node (P2s) at (3,-1) {$P_2 \rightarrow 0$};
		\node (P1) at (-2,-2) {$0 \rightarrow P_1$}; \node (S2) at (0,-2) {$P_1 \xrightarrow{\alpha} P_2$}; \node (I3) at (2,-2) {$P_2 \xrightarrow{\beta} P_3$}; \node (P3s) at (4,-2) {$P_3 \rightarrow 0$};
		
		\draw[->] (P1) --(P2); \draw[->] (P2) --(P3); \draw[->] (P3) --(I2); \draw[->] (I2) --(I3); \draw[->] (P2) --(S2); \draw[->] (S2) --(I2); \draw[->] (I2) -- (P1s); \draw[->] (I3) -- (P2s); \draw[->] (P1s) -- (P2s); \draw[->] (P2s) -- (P3s);
	\end{tikzpicture} .
\end{center}

In Table~\ref{A3}, we show all silting objects, their respective cotorsion pairs, thick subcategories, wide subcategory and torsion class given by the bijections in \cref{main-coro}. The dots correspond to the objects depicted in the AR quiver of $\K_{\Bbbk  Q}$ (or $\Mod \Bbbk Q$), and the shaded areas correspond to the subcategory additively generated by the dots they contain. In the second column, the blue shaded area in each figure depict the subcategory $\X$, while the orange shaded area plays the role of $\Y$ for the cotorsion pairs $(\X, \Y)$ they illustrate. 
\end{example}

\input{figures.tex}

\begin{example}[\textbf{$\Ww$ is not a bijection in general}]
Consider now the Kronecker quiver \[Q = \begin{tikzcd} 1 \arrow[r,shift left=.75ex,"\alpha"] \arrow[r,shift right=.75ex,"\beta", swap] & 2 \end{tikzcd}. \]
Let $\C$ be thick subcategory of $\K_{\Bbbk Q}$ whose objects are the projective presentations of regular modules. In particular, any object $X \in \C$ satisfies that $[X] = n[P_1] - n[P_2] = (n, -n)$ for some $n \in \mathbb{Z}_{> 0}$. Let $M \neq 0$ be an indecomposable module in $\Ww(\C) \subset \Mod \Bbbk Q$ with minimal projective resolution $X_M \in \K_{\Bbbk Q}$. Since $\langle [X] , M\rangle = 0$ for all $X \in \C$, $M$ cannot be pre-projetive or pre-injective. If $M$ is regular, then $X_M \in \C$, but $M$ cannot be $X_M$-semistable. We conclude that $\Ww(\C) =\{0\} = \Ww(\K_{\Bbbk Q})$, and $\Ww$ is not a bijection in general. 
    
\end{example}

\section{Geometric interpretation}\label{s_geometry}

The goal of this section is to introduce and compare three notions of semistability in $\K_\Lambda = \K^{[-1,0]}(\proj \Lambda)$:
\begin{itemize}
	\item[$-$] $M$-semistability, defined by non-vanishing of the determinantal semi-invariant associated to $M \in \Mod \Lambda$ but easily interpreted in representation theoretical terms (\cref{defMsemistab1}, \cref{defMsemistab2});
	\item[$-$] geometric semistability, from GIT (\cref{geometricss1}, \cref{gemetricss2});
	\item[$-$] numerical semistability, directly inspired from the work of A. King (\cref{numericalss}).
\end{itemize}

That $M$-semistability implies geometric semistability is straightforward. We show that both notions imply numerical semistability (\cref{Mimplies[M]}), but that the converse does not hold in general (\cref{ex_numssnotgeoss}). Previous work by K.~Igusa, K.~Orr, G.~Todorov and J.~Weyman entails that geometric semistability implies $M$-semistability for acyclic algebras \cite[Theorem 6.4.1]{igusa2009cluster}. This result, together with \cref{wide-thick} and \cref{main-coro}, shows that $M$-semistability mirrors important properties of King's semistability for modules, namely the equivalence between this notion and the one arising from GIT, the fact that categories of semistable modules are wide, and its intricate relation to $\tau$-tilting theory. The results of this section should be interpreted in the context of the tropical duality between g-vectors and c-vectors \cite{derksen2010quivers, najera2013cvectors, fu2017cvectors, treffinger2019cvectors}.

\vspace{2mm}
Throughout this section we fix $\theta^{-1}, \theta^0 \in \mathbb{Z}^n_{\geq 0}$ as well as $X^{-1} = \bigoplus_{i=0}^n P_i^{\theta_i^{-1}}$ and $X^0 = \bigoplus_{i=0}^n P_i^{\theta_{i}^0}$ two projective modules in $\Mod \Lambda$. Let  $R(X^{-1},X^0) = \Hom_{\Lambda}(X^{-1}, X^{0})$, $A(X^{-1},X^0) = \En(X^{-1}) \times \En(X^{0})^{op}$, and $G(X^{-1},X^0) = \Aut(X^{-1})^{op} \times \Aut(X^{0}) \subset A(X^{-1},X^0)$. The group $G(X^{-1},X^0)$ acts on the affine space $R(X^{-1},X^0)$ via simultaneous multiplication: let $g = (g_{-1},g_0) \in G(X^{-1},X^0)$, where $g_{-1} \in \Aut(X^{-1})$, $g_0 \in \Aut(X^0)$ and $x : X^{-1} \rightarrow X^0$, then $g \cdot x = g_0 \cdot x \cdot g_{-1}$. Note that this is a well defined action since we chose to work in $\Aut(X^{-1})^{op}$. If we were to consider $\Aut(X^{-1})$ instead, some sign conventions would have to be adjusted. When the context allows it, we will write $R$ instead of $R(X^{-1}, X^0)$ and do the same for the groups $A$ and $G$.

Since we consider modules over a finite-dimensional $\Bbbk$-algebra $\Lambda$, $A = A(X^{-1},X^0)$ is also finite-dimensional and thus its radical $N$ coincides with its nil-radical. By the Wedderburn-Artin theorem, we have that 
$$ A/N \cong \bigoplus_{i = 1}^n \left( M_{\theta_i^{-1}}(D_i) \times M_{\theta_i^0}(D_i)\right)$$
where $N = \rad(\En(X^{-1})) \times \rad(\En(X^{0}))$ and $D_i = \En(P_i)/\rad(\En(P_i)) \simeq \Bbbk$. 
Using the fact that $f \in A$ is invertible if and only if its image in $A/N$ is invertible, we get that 
$$ G = G(X^{-1},X^0) = (1_A + N) \rtimes \left(\prod_{i = 1}^n \left( GL_{\theta^{-1}_i}(D_i) \times GL_{\theta^0_i}(D_i)\right)\right).$$
Let $U = 1_A + N$, then $U$ is a normal subgroup of $G$ and, by definition, all its elements are unipotent. Since $\Lambda$ is finite-dimensional, there exists $m \in \mathbb{Z}_{>0}$ and a series of subgroups
$$0 \unlhd 1_A + N^m \unlhd \cdots \unlhd 1_A + N^2 \unlhd 1_A + N  = U$$ where $1+ N^i$ is a normal subgroup of $1 + N^{i-1}$, and such that every quotient is abelian. This shows that $U$ is solvable and hence, is the unipotent radical of $G$, since it is closed and connected. 
Any algebraic group $G$ over $\mathbb{C}$ satisfies that $G = U \rtimes G_{red}$ where $U$ is its unipotent radical and $G_{red}$ is reductive. If $\Bbbk = \mathbb{C}$, we then get that $ G(X^{-1}, X^0)_{red} \cong \prod_{i = 1}^n \left( GL_{\theta^{-1}_i}(D_i) \times GL_{\theta^0_i}(D_i)\right)$, where $D_i \cong \mathbb{C}$. From now on, we suppose that $\Bbbk \simeq \mathbb{C}$.

A \textbf{character} of $G$ is a morphism of algebraic groups $\chi : G \rightarrow \mathbb{C}^*$. If $G$ is unipotent, every character is trivial. In particular, when $G =  U \rtimes G_{red}$, the set of characters of $G$ can be identified with that of $G_{red}$. In our case, every character $\chi : \prod_{i = 1}^n \left( GL_{\theta^{-1}_i}(\mathbb{C}) \times GL_{\theta^0_i}(\mathbb{C})\right) \rightarrow \mathbb{C}^*$ is given by $\chi\left((g_{-1}^i,g_{0}^i)_{1 \leq i \leq n} \right) = \prod_{i=1}^{n} \det(g_{-1}^i)^{d_i^{-1}} \cdot \det(g_{0}^i)^{d_i^{0}}$, and thus, the group of characters of $G$ is isomorphic to $\mathbb{Z}^{2n}$. For $\bar{d} = (d^{-1}, d^0) \in \mathbb{Z}^{2n}$ we will write $\chi_{\bar{d}}$ for the character given by the previous formula.  

Dually, a group morphism $\lambda : \mathbb{C}^* \rightarrow G$ is called a \textbf{one-parameter subgroup} or a \textbf{co-character} of $G$. They are all of the form $u \hat{\lambda} u^{-1}$ for $u \in U$ and $\hat{\lambda} : \mathbb{C}^* \rightarrow G_{red}$. Indeed, let $H = ker(\lambda) = \lambda^{-1}(1_G)$. $H$ is a closed subgroup of $\mathbb{C}^*$, hence, $\lambda(\mathbb{C}^*) \cong \mathbb{C}^* / H$ is a reductive subgroup of $G = U \rtimes G_{red}$ since $\mathbb{C}^*$ is. As we are working in characteristic 0, by~\cite[Proposition 4.2]{hochschild2012basic}, there exists $u \in U$ such that $u^{-1}\lambda(\mathbb{C}^*)u$ $\leq G_{red}$. Then $\hat{\lambda} = u^{-1}\lambda u : \mathbb{C}^* \rightarrow G_{red}$ satisfies the property. Moreover, since $G_{red}$ is a product of $GL_k(\mathbb{C})$'s, all one parameter subgroups are of the form $ \lambda = u \tilde{\lambda} u^{-1},$ where $\tilde{\lambda}$ is a one-parameter subgroup with image in a maximal torus of $G_{red}$.

The composition $\chi \circ \lambda : \mathbb{C}^* \rightarrow \mathbb{C}^*$ gives us a paring $\langle -,- \rangle$ between  the set of one parameter subgroups and the character group. Indeed, since every algebraic group automorphism of $\mathbb{C}^*$ is of the form $t \mapsto t^m$ for some $m \in \mathbb{Z}$, we define $\langle \lambda, \chi \rangle$ to be the integer $m$ such that $\chi \circ \lambda (t) = t^m$.

\subsection{Geometric Invariant Theory}\label{ss_git}
Geometric Invariant Theory (GIT) was developed by D. Mumford as a method for constructing quotients by group actions on algebraic varieties. One of the key tools in this theory is the notion of semi-invariant and semistability. In this section we study those tools for the action of $G(X^{-1}, X^0)$ over the vector space $R(X^{-1}, X^0)$ and prove a technical result that will be used in the rest of the section. For more on the generalities of GIT, see~\cite{mumford1994geometric}. 

Let $R= R(X^{-1}, X^0)$ and $G = G(X^{-1}, X^0)$. A \textbf{semi-invariant} $f \in \mathbb{C[R]}$ of weight $\chi \in \Hom(G, \mathbb{C}^*)$ is a regular function such that $f(g \cdot x) = \chi(g) f(x) $ for all $g \in G$ and all $x \in R$. For a non-trivial character $\chi$, define the graded ring $$SI(R)^{G,\chi} = \bigoplus_{n \in \mathbb{Z}_{\geq 0}} \mathbb{C}[R]^{G,\chi^n},$$ where the $\mathbb{C}[R]^{G,\chi^n}$ denote the set of semi-invariant functions over $R$ of weight $\chi^n$. 

\begin{definition}\label{geometricss1}
	Let $x \in R$ and $\chi$ a character of $G$. We say that $x$ is \textbf{$\chi$-semistable} if there exists $m \in \mathbb{Z}_{> 0}$ and $f \in \mathbb{C}[R]^{G,\chi^m}$ such that $f(x) \neq 0$. We denote by $R^{\chi, ss}$ the open subset of semistable points. 
\end{definition}

The points of the scheme $\text{Proj}(SI(R)^{G,\chi})$ should correspond (up to GIT- equivalence) to orbits of $\chi$-semistable points. However, since in our setting $G$ is not reductive in general, the ring $SI(R)^{G,\chi}$ of semi-invariants is not necessarily finitely generated, and thus $\text{Proj}(SI(R)^{G,\chi})$ is not a variety in general, even as $R$ is an affine space. Moreover, the Hilbert-Mumford numerical criterion to determine weather a point $x$ is $\chi$-semistable does not necessarily hold. However, we seek to describe features of the characters for which semistable points exist.

Let $G_0 = \{g \ | \ g \cdot x = x \quad \forall x \in R \}$, and suppose that $x \in R$ is $\chi_{\bar{d}}$-semistable for some $\bar{d} \in \mathbb{Z}^{2n}$. Let $f \in \mathbb{C}[R]^{G, \chi^m_{\bar{d}}}$ such that $f(x) \neq 0$ with $m \geq 1$, then $f(g \cdot x ) = \chi(g)^m_{\bar{d}} f(x) = f(x)$ for any $g \in G_0$. That is $\chi^m_{\bar{d}}(G_0) \equiv 1$. In particular for $ \Delta = \left\{ (t^{-1} \cdot \Id_{\theta^{-1}_i}, t \cdot \Id_{\theta^0_i}) \ | \ t \in \mathbb{C}^* \right\}_{1 \leq i \leq n} \cong \mathbb{C}^* \subseteq G_0$, we must have that 

\begin{gather*}\chi^m(\Delta) = \left( \prod_{i =1}^n \det(t^{-1} \cdot \Id_{\theta^{-1}_i})^{d_i^{-1}} \det(t \cdot \Id_{\theta^0_i})^{d_i^{0}} \right)^m = \\ t^{m\left( -\sum_{i=1}^n  \theta^{-1}_i d_i^{-1} + \sum_{i=1}^n  \theta^0_i d_i^{0} \right)} = 1 \end{gather*}

which in turn implies that  $$-\sum_{i=1}^n  \theta_i^{-1} d_i^{-1} + \sum_{i=1}^n  \theta^0_i d_i^{0} = \langle(-[X^{-1}],[X^0]) ,(d^{-1},d^0) \rangle= \langle (-[X^{-1}], [X^0]), \chi \rangle = 0.$$

Let $x \in R = R(X^{-1},X^0)$ and consider its orbit $G \cdot x \subset R$. Note that $G \cdot x$ can be identified with the isomorphism class of $x$ as an object in $C^{[-1,0]}(\proj\Lambda)$. The following proposition will gives us a link between inflations in $C^{[-1,0]}(\proj\Lambda)$ and semistability. 

\begin{proposition}\label{naive}
	Let $x \in R$ with associated $X =$ \begin{tikzcd}[cramped, sep=small]X^{-1} \arrow[d, "x"] \\ X^0 \end{tikzcd} $\in C^{[-1,0]}(\proj\Lambda)$. If $x$ if $\chi$-semistable, then 
	\begin{enumerate}
		\item $\langle (-[X^{-1}], [X^0]), \chi \rangle = 0$;
		\item For any inflation $Y \rightarrowtail X$ in $C^{[-1,0]}(\proj\Lambda)$, we must have that $$\langle (-[Y^{-1}], [Y^0]), \chi \rangle \geq 0.$$
	\end{enumerate}
\end{proposition}

\begin{proof}
	Suppose $x$ is $\chi$-semistable with respect to the $G$-action. By the previous discussion we have (1). Let $f$ be a $\chi^m$ a semi-invariant for $m \geq 1$ such that $f(x) \neq 0$. If $\lambda$ is a one-parameter subgroup of $G$, we must have that for every $t \in \mathbb{C}^*$,
	$$f(\lambda(t) \cdot x) = \chi^m(\lambda(t)) f(x) = t^{m \langle \lambda, \chi \rangle} f(x).$$
	Suppose that $\lim_{t\to 0} \lambda(t) \cdot x$ exists and it is equal to $x' \in R$, then $$f(x') = \lim_{t\to 0} t^{m \langle \chi, \lambda \rangle} f(x).$$ Since $f(x) \neq 0$, we must have that $\langle \lambda, \chi \rangle \geq 0$. The statement in (2) will follow from noting that one-parameter subgroups such that $\lim_{t\to 0} \lambda(t) \cdot x$ exists correspond to inflations of $X$ in $\C^{[-1,0]}(\proj \Lambda)$. As we have seen before, $ \lambda(t) = u \tilde{\lambda}(t) u^{-1}$ where $\tilde{\lambda}$ is a one-parameter subgroup with image in a maximal torus of $G$. Explicitly, 
	
	\begin{gather*}
		\lambda(t) = (\lambda_{-1}(t), \lambda_{0}(t)) = (g_{-1} \ \tilde{\lambda}_{-1}(t) \ (g_{-1})^{-1}, g_0 \ \tilde{\lambda}_{0}(t) \ g_0^{-1}) = \\ \displaybreak[0] = \left(g_{-1} \begin{psmallmatrix}
			{\begin{psmallmatrix}
					t^{\lambda^{-1}_{1,1}} & \cdots & 0 \\ 
					\vdots & \ddots & \vdots \\
					0 & \cdots & t^{\lambda^{-1}_{\theta^{-1}_1,1}}
			\end{psmallmatrix}} & 0 & \cdots & 0 &0 \\  
			\vdots & \ddots & \vdots & \ddots &\vdots \\  
			0 & \cdots & {\begin{psmallmatrix}
					t^{\lambda^{-1}_{1,i}} & \cdots & 0 \\ 
					\vdots & \ddots & \vdots \\
					0 & \cdots & t^{\lambda^{-1}_{\theta^{-1}_i,i}}
			\end{psmallmatrix}} & \cdots & 0\\ 
			\vdots & \ddots & \vdots & \ddots &\vdots \\  
			0 & 0 & \cdots & 0 & {\begin{psmallmatrix}
					t^{\lambda^{-1}_{1,n}} & \cdots & 0 \\ 
					\vdots & \ddots & \vdots \\
					0 & \cdots & t^{\lambda^{-1}_{\theta^{-1}_n,n}}
			\end{psmallmatrix}}  
		\end{psmallmatrix} (g_{-1})^{-1}  , \right. \\ \displaybreak[0] \left. \quad g_0 \begin{psmallmatrix}
			{\begin{psmallmatrix}
					t^{\lambda^0_{1,1}} & \cdots & 0 \\ 
					\vdots & \ddots & \vdots \\
					0 & \cdots & t^{\lambda^0_{\theta^0_1,1}}
			\end{psmallmatrix}} & 0 & \cdots & 0 &0 \\
			\vdots & \ddots & \vdots & \ddots &\vdots \\
			0 & \cdots & {\begin{psmallmatrix}
					t^{\lambda^0_{1,i}} & \cdots & 0 \\ 
					\vdots & \ddots & \vdots \\
					0 & \cdots & t^{\lambda^0_{\theta^0_i,i}}
			\end{psmallmatrix}} & \cdots & 0\\
			\vdots & \ddots & \vdots & \ddots &\vdots \\
			0 & 0 & \cdots & 0 & {\begin{psmallmatrix}
					t^{\lambda^0_{1,n}} & \cdots & 0 \\ 
					\vdots & \ddots & \vdots \\
					0 & \cdots & t^{\lambda^0_{\theta^0_n,n}}
			\end{psmallmatrix}}
		\end{psmallmatrix} g_0^{-1} \right)
	\end{gather*}

	where $\lambda_{-1}(t), \ g_{-1} \in \Aut(X^{-1})$, and $\lambda_{0}(t), \ g_0 \in \Aut(X^{0})$ for every $t \in \mathbb{C}^*$. Here, $\lambda_{l,i}^{\varepsilon}$ is the weight corresponding to the $l$-th copy of the projective indecomposable $P_i$ inside of $X^\varepsilon$ with $\varepsilon \in \{-1, 0 \}$, $1 \leq i \leq n$ and $1 \leq l \leq \theta^\varepsilon_i$. We get that, for $\varepsilon \in \{-1, 0 \}$, $X^\varepsilon = \bigoplus_{m \in \mathbb{Z}} X_m^i$  where each $X_m^i$ is the direct sum of indecomposable projectives summands $Q$ of $X^\varepsilon$ such that $\lambda_\varepsilon(t)(Q) = t^m Q$. So, for any $m, n \in \mathbb{Z}$, we have the following commutative diagram : 
	
	\[\begin{tikzcd}[row sep = large, column sep= huge]
		X_n^{-1} \arrow[d, "t^{n+m} \left(\pi_m^0 \cdot \left.x\right|_{X_n^{-1}} \right)"] \arrow[r, hookrightarrow] & X^{-1\arrow[d, "\lambda_{0}(t) \cdot  x \cdot \lambda_{-1}(t)"]} \\
		X_m^{0} & X^0 \arrow[l,twoheadrightarrow, "\pi_{m}^0"]\\
	\end{tikzcd}\]
	
	Since the limit when $ t \to 0$ exists, then $\pi_m^0 \cdot \left. x\right|_{X_n^{-1}}$ must be zero when $n+m < 0$. Let $X_{\leq n}^{-1} = \bigoplus_{i \leq n} X_{-i}^{-1}$ and $X_{\leq n}^0 = \bigoplus_{i \leq n} X_{i}^{0}$. Then, for every $n \in \mathbb{Z}$, $x$ defines the objects $X_{\leq n}:=$\begin{tikzcd}[cramped, sep=small]X_{\leq n}^{-1} \arrow[d, "x_{\leq n}"] \\ X_{\leq n}^0 \end{tikzcd}, which make the following diagram commute
	\[\begin{tikzcd}[column sep = large]
		X^{-1}_{\leq n} \arrow[r, tail] \arrow[d, "x_{\leq n}"] & X^{-1} \arrow[d, "x"] \\
		X^{0}_{\leq n} \arrow[r, tail] & X^0 
	\end{tikzcd}\]
	Here $x_{\leq n} = \pi_m^0 \cdot \left.x\right|_{X_n^{-1}}$ when $m+n \geq 0$ and equals 0 otherwise. This gives us a sequence $0 \rightarrowtail \cdots \rightarrowtail X_{\leq n} \rightarrowtail X_{\leq n+1} \rightarrowtail \cdots \rightarrowtail X$ of inflations for $X$. Note that $[ X^{i} ]=[ \bigoplus X^{i}_{\leq n+1} / X^{i}_{\leq n} ] = \sum_{n \in \mathbb{Z}} [X^{i}_{\leq n+1} / X^{i}_{\leq n}]$ for $i \in \{ -1, 0\}$. For a projective module $Q$ and $1 \leq j \leq n$, denote by $[Q]_j$ the number of times the indecomposable projective $P_j$ appears as a direct summand of $Q$. We can express the value of $\langle \lambda, \chi \rangle$ as 
	\begin{gather*}
		\langle \lambda , \chi \rangle = \\ \displaybreak[0] \left( \sum_{j= 1}^n \theta_j^{-1} \left(\sum_{m \in \mathbb{Z}}(-m) [X^{-1}_{\leq m} / X^{-1}_{\leq m-1}]_j\right)  \right) + \left( \sum_{j= 1}^n \theta_j^{0} \left(\sum_{m \in \mathbb{Z}}m [X^{0}_{\leq m} / X^{0}_{\leq m-1}]_j\right)  \right) = \\ \displaybreak[0] 
		= \sum_{m \in \mathbb{Z}} m \langle \left( -[X^{-1}_{\leq m} / X^{-1}_{\leq m-1}] , [X^{0}_{\leq m} / X^{0}_{\leq m-1}] \right) , \chi \rangle = \\
		= \sum_{m \in \mathbb{Z}} \langle \left(-[X^{-1}_{\leq m}] ,  [X^{0}_{\leq m}] \right), \chi \rangle.
	\end{gather*}
	That is, the value of the paring between $\chi$ and $\lambda$ is given by the inner product between the associated integer vector of $\chi$ and the classes in $\mathcal{K}_0(\proj\Lambda)$ of the projective modules $X^i_{\leq m}$ for $i \in \{-1, 0\}$. 
	
	Given an object $Y = $\begin{tikzcd}[cramped, sep=small]Y^{-1} \arrow[d, "y"] \\ Y^0 \end{tikzcd} that is the source of an inflation to $X$, we construct a one-parameter subgroup $\lambda_{Y}$ such that it's associated filtration is $0 \hookrightarrow Y \hookrightarrow X$. Suppose $Y^{-1} = \bigoplus_{i=0}^n P_i^{\theta'^{-1}_i}$ and $Y^0 = \bigoplus_{i=0}^n P_i^{{\theta'_i}^0}$ with $\theta'^0_i \leq \theta^0_i$ and $\theta'^{-1}_i \leq \theta^{-1}_i$ for all $1 \leq i \leq n$. Up to isomorphism, we can suppose $Y^{i} \subseteq X^{i}$. Let $$\lambda_{Y}(t) = \left( \prod_{i = 1}^n \text{diag}_{\theta^{-1}_i}(1, \cdots 1, t^{-1}, \cdots t^{-1}), \prod_{i = 1}^n \text{diag}_{\theta^0_i}(1, \cdots 1, t, \cdots t) \right)$$ where each diagonal matrix has $\theta'^{-1}_i$ and $\theta'^0_i$ 1's respectively. Then $X^{-1}_{\leq 0} = Y^{-1}$, $X^{0}_{\leq 0} = Y^{0}$,  $X^{-1}_{\leq i} = X^{0}_{\leq i} = 0$ for all $i < 0$ and $X^{-1}_{\leq i} = X^{-1}$ and $X^{0}_{\leq i} = X^{0}$ for all $i > 0$. Since $x$ is semistable, 
	\begin{gather*}
		\langle \lambda_{Y}, \chi \rangle  = \sum_{i < 0}\langle \left(-[X^{-1}_{\leq i}] ,  [X^{0}_{\leq i}] \right), \chi \rangle + \langle \left(-[Y^{-1}] ,  [Y^{0}] \right) , \chi \rangle + \\ \sum_{i > 0}\langle \chi, \left(-[X^{-1}_{\leq i}] ,  [X^{0}_{\leq i}] \right) \rangle
		= 0 + \langle (-[Y^{-1}], [Y^0]), \chi \rangle + 0 \geq 0
	\end{gather*}
\end{proof}

\begin{remark}
	Note that if every inflation satisfies (2) from the previous proposition, then $\langle\lambda, \chi \rangle \geq 0$ for every one-parameter subgroup $\lambda$ such that the limit  $\lim_{t\to 0} \lambda(t) \cdot x$ exists. If $G$ were reductive, this would imply that $x$ is $\chi$-semistable as in \cite{King}.
\end{remark}

\subsection{Determinantal invariants}\label{ss_det}
To study whether \cref{naive} has a partial converse, one could try to explicitly describe its ring of semi-invariants $SI(R)^{G, \chi}$ for a given character $\chi$.

\begin{definition}[\cite{schofield1991semi}]
	Let $X^{-1}, X^0 \in \proj \Lambda$ and $M \in \Mod \Lambda$ such that $~\langle [X^0]-[X^{-1}], [M] \rangle =0$, that is, such that $\dim_\Bbbk(\Hom(X^{-1}, M)) = \dim_\Bbbk(\Hom(X^{0}, M)$. Let $R(X^{-1}, X^0)$ be as defined in the beginning of \cref{s_geometry}. We denote by $s(-,M)$ the regular function such that, for any $x \in R(X^{-1}, X^0)$ $$s(x,M) = \det \left(\Hom (X^0, M) \xrightarrow{- \circ x = x^*} \Hom(X^{-1}, M)\right).$$ We say that the map $s(-,M)$ is a \textbf{determinantal semi-invariant}. 
\end{definition}
\begin{remark}
	Let $M \in \Mod \Lambda$ and let $\dim M$ be its class in $K_0(\Mod \Lambda)$. Note that the value of the map $s(-,M)$ depends on the choice of basis for the $\Hom (X^0, M)$ and $\Hom(X^{-1}, M)$ spaces. Once we fix a basis for $\Hom(P_i, S_j)$ for all $1 \leq i, j \leq 0$, we know that for any $P \in \proj \Lambda$ with a given decomposition $P = \bigoplus_{i = 1}^n P_i^{e_i} \in \proj \Lambda$ and $M \in \Mod \Lambda$, then 
	$$\Hom(P, M) = \bigoplus_{i=1}^n \bigoplus_{j=1}^i \Hom(P_i, S_j)^{e_i (\dim M)_j}$$
	and thus basis is given by that of the $\Hom(P_i, S_j)$. However, we mostly care about the non-annihilation of these functions, and since base change doesn't affect this property, the choice of basis is mostly omitted. 
\end{remark}
\begin{proposition}~\cite[Proposition 5.13]{igusa2009cluster}\label{determinv}
	Let $X^{-1}, X^0 \in \proj \Lambda$ and $M \in \Mod \Lambda$. The regular function $s(-,M)$ is a $G(X^{-1}, X^0)$ semi-invariant over $R(X^{-1}, X^0)$ with associated character $\chi_{([M], [M])}$. 
\end{proposition}
\begin{proof}
	By hypothesis $\dim_\Bbbk \Hom(X^0,M) = \dim_\Bbbk \Hom(X^{-1},M)$, thus $s(s,M)$ is well defined for any $x \in R(X^{-1}, X^0)$. Let $g = (g_{-1}, g_0) \in G(X^{-1}, X^0) = \Aut(X^{-1}) \times \Aut(X^0)$, then 
	\begin{gather*}
		s(g \cdot x, M) =  \det \left(\Hom (X^0, M) \xrightarrow{(g_0 \cdot x \cdot g_{-1})^*} \Hom(X^{-1}, M)\right) = \\ = \det(g_0^*) \cdot s(s,M) \cdot \det(g_{-1}^*).
	\end{gather*}
	The regular function $\chi(g) = \det (g_0^*) \cdot \det(g_{-1}^*)$ defines a character for the action of $G(X^{-1}, X^0)$, and as such, it factors through $G(X^{-1}, X^0)_{red}$. Recall $(g_0)_{red} = (g_{0}^i) \in \prod^n_{i = 1} GL_{\theta_{i}^0}(D_i)$. Since $\Hom(P_i,M) \cong M_i$ $\forall \ 1 \leq i \leq n$, where $M_i$ is the vector space in the vertex $i$ associated to $M$, we have that $(g_0)_{red}^*$ is a bloc-diagonal matrix in which the bloc corresponding to $g_{0}^i$ appears $\dim M_i$ times. Thus, $\det (g_0^*)_{red} = \prod_{i=1}^{n} \det(g_{0}^i)^{d_i}$ where $ \dim M = (\dim M_i)_{1 \leq i \leq n} = (d_i)_{1 \leq i \leq n}$ is the dimension vector of $M$. The same argument gives $\det (g_{-1})_{red} =  \prod_{i=1}^{n} \det(g_{-,i})^{d_i}$ and so $s(g \cdot X, M) = \chi(g) \cdot s(X,M)$ where $\chi$ is of weight $([M], [M])$. 
\end{proof}

\begin{remark}\label{detervirtual}
	Let $X^{-1}$, $X^0$ and $M$ be as before. Consider now $R(X^{-1}\oplus P, X^0 \oplus P)$ for some $0 \neq P \in \proj \Lambda$. Since $\langle \left[X \oplus \text{\begin{tikzcd}[cramped, sep=tiny]  P \arrow[d, equal]\\ P \end{tikzcd}}\right], [M] \rangle = \langle [X], [M]\rangle = 0$, there is a regular function $s'(-,M)$ that is a semi-invariant for the action of $G(X^{-1}\oplus P, X^0 \oplus P)$ on $R(X^{-1}\oplus P, X^0 \oplus P)$ and satisfies that $$s'\left(\begin{pmatrix} - & 0 \\ 0 & \Id_P \end{pmatrix}\right) = s(-,M)$$ for any $x \in R(X^{-1}, X^0)$, where $s(-,M)$ is as in \cref{determinv}. Let $x' \in R(X^{-1}\oplus P, X^0 \oplus P)$ and suppose that it belongs to the orbit of the point $\begin{pmatrix} x & 0 \\ 0 & \Id_P \end{pmatrix}$ for some $x \in R(X^{-1}, X^0)$. Then, there exists $g \in G(X^{-1}\oplus P, X^0 \oplus P)$ such that
	\begin{gather*}
		s'(x',M) = s'\left( g \cdot {\begin{pmatrix} x & 0 \\ 0 & \Id_P \end{pmatrix}} \right) = \chi_{([M], [M])}(g) \cdot s'\left( {\begin{pmatrix} x & 0 \\ 0 & \Id_P \end{pmatrix}} \right) = \\ = \chi_{([M], [M])}(g) \cdot s(x, M)
	\end{gather*}
	Thus, $s'(x',M) \neq 0$ if and only if $s(x,M) \neq 0$.
\end{remark}

We now shift our attention back to $\K_\Lambda = \K^{[-1,0]}(\proj \Lambda)$. The last remark tells us that for any $X \in \K_\Lambda$ and $M \in \Mod \Lambda$ such that $\langle [X], [M] \rangle = 0$, the non-annihilation of the $s(-,M)$ does not depend on the representative of $X$ in $\C^{[-1, 0]}(\proj \Lambda)$. We rephrase \cref{defMsemistab1} in terms of the $s(-,M)$.

\begin{definition}[\textbf{$M$-semistability}]\label{defMsemistab2}
	Let $X \in \K^{[-1,0]}(\proj \Lambda)$ and $M \in \Mod \Lambda$. We say that $X$ is \textbf{$M$-semistable} if $\langle [X], [M]\rangle = 0$ and if there exist $x \in R(X^{-1}, X^0)$  such that $X \simeq \text{\begin{tikzcd}[cramped, sep=small]  X^{-1} \arrow[d, "x"] \\ X^0 \end{tikzcd}}$ and $s(x, M) \neq 0$.
\end{definition}

These semi-invariants and their links to cluster algebras where thoroughly studied by K.~Igusa, K.~Orr, G.~Todorov and J.~Weyman in \cite{igusa2009cluster,igusa2015modulated}. In their work, they define the ring of \textbf{virtual semi-invariants} for any $\theta \in \K_0(\proj \Lambda)$ (see \cref{ssforpp} for the definition), which can be interpreted as those semi-invariants that are well defined for objects in $\K_\Lambda$. Notably, they show that that when $\Lambda \simeq \Bbbk Q$, where $Q$ is a finite quiver without oriented cycles, then the ring of virtual semi-invariants is spaned by determinantal semi-invariants \cite[Theorem 6.4.1 (Virtual First Fundamental Theorem)]{igusa2009cluster}. Although we are interested in the question of whether this holds for a general finite-dimensional algebra $\Lambda$, the goal of this paper is to find a new categorical significance to this semi-invariants, inspired by all the theory branching off semistability theory in $\Mod \Lambda$.

\subsection{Semistability in $\Mod \Lambda$}\label{ss_stabmod}
Recall that that the notion of King's semistability on $\Mod\Lambda$, first introduced in~\cite{King}, gives rise to a certain class of wide subcategories of $\Mod \Lambda$. 
\begin{definition}\cite{King}\label{def-osem}
	Let $\theta \in K_0(\K_\Lambda)$, we say that $M \in \Mod\Lambda$ is \textbf{$\theta$ - semistable} if 
	\begin{enumerate}
		\item $\langle \theta, [M] \rangle = 0$;
		\item For every submodule $N \subset M$, $\langle \theta, [N] \rangle \leq 0$.
	\end{enumerate}
	We denote by $\Ww_{\theta} \subset \Mod \Lambda$ the wide subcategory whose objects are those who are $\theta$ - semistable.
\end{definition}
One of the key results in~\cite{King} is that, for a module $M$ of dimension $d$, the notion of being $\theta$-semistable is equivalent to the existence of a semi-invariant $f$ of weight $n \theta$ for some $n \in \mathbb{Z}_{\geq 0}$ over the variety of representations $\text{rep}(\Lambda, d)$ such that $f(M) \neq 0$. When $\Lambda$ is a finite-dimensional algebra over an algebraic closed field, all of these semi-invariants are generated by the rational functions $s(X,-)$ where $X \in \K_\Lambda$ such that $[X] = n \theta$ for some $n \in \mathbb{Z}_{\geq 0}$ \cite{King,derksen2000semi,schofield2001semi,domokos2002relative}. That is, $M$ is $\theta$-semistable if and only if there exists $X \in \K_\Lambda$ such that $s(X,M) \neq 0$. We rephrase this statement in terms of the subcategories $\Ww(X) \subset \Mod \Lambda$ defined in \cref{stabinK}. 

\begin{proposition}\label{unionofwide} \cite{King,derksen2000semi,schofield2001semi,domokos2002relative}
	Let $\theta \in K_0(\K_\Lambda)$, then $$\Ww_{\theta} = \bigcup_{\substack{X \in \K_\Lambda \\ [X] = \mathbb{Z}_{\geq 0} \theta}} \Ww(X) \ .$$
\end{proposition}
\begin{proof}
	If $M$ is $\theta$-semistable, there exists $X =$ \begin{tikzcd}[cramped, sep=small]X^{-1} \arrow[d, "x"] \\ X^0 \end{tikzcd} $\in \K_\Lambda$ such that $[X]=n \theta$ for some $n \in \mathbb{Z}_{\geq 0}$ and $s(X,M) \neq 0$, which in turns translates to $x^*$ being an isomorphism. We get that for every $M \in \Ww_{\theta}$, there is $X$ such that $M \in \Ww(X)$, and so $\Ww_{\theta} \subset \bigcup_{\substack{X \in \K_\Lambda \\ [X] = \mathbb{Z}_{\geq 0} \theta}} \Ww(X)$. The other inclusion follows from the fact that $s(X,-)$ is a semi-invariant of weight $[X] = n\theta$ for some $n \in \mathbb{Z}_{\geq 0}$.
\end{proof}

The following result shows the relation between semistability in $\Mod \Lambda$ and $\tau$-tilting theory. It has been proved in full generality in \cite{yurikusa2018wide, brustle2019wall}. We include here a proof when $\Lambda$ is an algebra over a algebraically closed field, to showcase the relevance of the $s(X,-)$ invaraints, and why one could be brought to define semistability in their terms. 

\begin{theorem}~\cite{yurikusa2018wide, brustle2019wall}\label{widesemistable}
	Let $U \in \K_\Lambda$ be a presilting complex. Then $\Ww_{[U]} = {}^{\perp}H^{-1}(\nu U) \cap H^0(U)^{\perp} = \Ww(U)$ where $\nu$ is the Nakayama fuctor.
\end{theorem}
\begin{proof}
	Let $U \in K_\Lambda$ be a presilting complex. By \cref{unionofwide}, we have that $\Ww(U) \subset \Ww_{[U]}$. Now consider $X \in \K_\Lambda$ such that $[X] = n [U]$ for some $n \in \mathbb{Z}_{\geq 0}$ and let $M \in \Ww(X)$. Then, there exist $x \in R(X^{-1}, X^0) = R$ such that $X \simeq $ \begin{tikzcd}[cramped, sep=small]X^{-1} \arrow[d, "x"] \\ X^0 \end{tikzcd} and $s(x,M) \neq 0$. Since this condition is open, the generic point $\eta \in R$ must satisfy that $s(\eta, M) \neq 0$. Because $U$ is presilting, we can choose a representative \begin{tikzcd}[cramped, sep=small]U^{-1} \arrow[d, "u"] \\ U^0 \end{tikzcd}, such that there are no non-zero common projective direct summands between $U^{-1}$ and $U^0$. In particular, we can suppose that there is $P \in \proj \Lambda$ such that $X^i = (U^i)^{\oplus n} \oplus P$ for $i\in \{-1, 0\}$. Let $s'(-,M)$ be the determinantal semi-invariant definided by $M$ on $R' = R((U^{-1})^{\oplus n}, (U^0)^{\oplus n})$, by~\cite[Corollary 6.2.2]{igusa2009cluster} we get that $s'(\eta', M) \neq 0$ since $s(\eta, M) \neq 0$, where $\eta'$ is the generic point in $R'$. By a Dehy-Keller argument, we know that since $U^{\oplus n}$ is a 2-term presilting complex, the orbit $\mathcal{O}_{u^{\oplus n}}$ inside $R'$ must be open and dense. In particular, $\mathcal{X} \cap \mathcal{O}_{u^{\oplus n}} \neq \emptyset$, where $\mathcal{X} = \{y \in R' \ | \ s'(y, M) \neq 0\}$. Then, there must exist $u' \in \mathcal{O}_{u^{\oplus n}}$ such that $s(u',M) \neq 0$. Since $U$ is a direct summand of $U^{\oplus n} \simeq$ \begin{tikzcd}[cramped, sep=small]{U^{-1}}^{\oplus n} \arrow[d, "u'"] \\ {U^0}^{\oplus n} \end{tikzcd}, we must have that $s(u,M) \neq 0$. We get that $\Ww(X) \subset \Ww(U)$ for all $X$ such that $[X] = n[U]$ for $n \in \mathbb{Z}_{\geq 0}$; and thus, $\Ww(U) = \Ww_{[U]}$. The rest of the proposition follows from the fact that we have an exact sequence 
	\begin{equation}\label{nu-silting}
		\begin{tikzcd}[cramped, sep=small] 
			\Hom(M, H^{-1}(\nu U)) \arrow[r] &\Hom(M, \nu U^{-1}) \arrow[r, "(\nu u)^*"] \arrow[d, "\simeq"] &\Hom(M, \nu U^0) \arrow[d, "\simeq"] & \\
			& D\Hom(U^{-1}, M) \arrow[r, "D(u^*)"] & D\Hom(U^0, M) \arrow[r] & D\Hom(H^0(U), M),
		\end{tikzcd}
	\end{equation}
	so that, if $s(U,M) \neq 0$ then $u^*$ must be an isomorphism and $\Hom(M, H^{-1}(\nu U)) = \Hom(H^0(U, M)) = 0$.
\end{proof}
\begin{remark}\label{remarkonsilting}
	We always have the short exact sequence above, $U = \text{\begin{tikzcd}[cramped, sep=small]  U^{-1} \arrow[d, "u"]\\ U^0 \end{tikzcd}}$ need not be presilting. This fact is only used to prove that $\Ww(U) = \Ww_{[U]}$, and it is not essential, we only need to suppose that either $u^{\oplus n}$ has an open dense orbit or that it is the generic point inside $\mathcal{U}$ for every $n \in \mathbb{Z}_{\geq 0}$. To get the short exact sequence (\ref{nu-silting}) only apply $\nu$ to the exact sequence $U^{-1} \xrightarrow{u} U^0 \rightarrow H^0(U)$ and use the fact that $\nu$ is right exact, that $\Hom(M,-)$ and $\Hom (-,M)$ are left-exact, and that for any projective module $\Hom(M,\nu P) \simeq D\Hom(P,M)$.
\end{remark}

\textit{Dehy-Keller argument:} Let $U = $\begin{tikzcd}[cramped, sep=small]U^{-1} \arrow[d, "u"] \\ U^0 \end{tikzcd} be any presilting object in $\K_\Lambda$. We only need to prove that the differential map $$\Aut(U^{-1})^{op} \times \Aut(U^0) \xrightarrow{\bar{u}} \Hom(U^{-1}, U^0)$$ given by $(g_{-1},g_0) \mapsto g_0 \cdot u - u \cdot g_{-1}$ surjective. We first recall that since $U$ is presilting, then it must be isomorphic to $P \oplus V^{-1} \xrightarrow{(0, v)} U^0$ where the complex $V = V^{-1} \xrightarrow{v} U^0$ is the minimal projective presentation of $H^0(U)$ and $\Hom(P, H^0(U))= 0$. Using the argument in section 2.1 of~\cite{dehy2008combinatorics} we know that the map $\Aut(V^{-1}) \times \Aut(U^0) \xrightarrow{\bar{v}} \Hom(V^{-1}, U^0)$ is surjective. Now let $f = (f_P, f') \in \Hom(P\oplus V^{-1}, U^0)$, then there exists $(g_{-1},g_0) \in \Aut(V^{-1}) \times \Aut(U^0)$ such that $f' = g_0 \cdot v - v \cdot g_{-1}$. Consider $\pi_V$ the cokernel of $v$, since $\Hom(P, H^0(U))= 0$, we have that the composition $\pi_V \cdot f_P = 0$ and $f_P$ restricts to $\Imm v$. But $V^{-1}$ is a projective cover for $\Imm v = \Ker \pi_v$, so there exists $h : P \rightarrow V^{-1}$ such that $f_P = v \cdot h$.  This gives us the result, since $\begin{pmatrix} \Id_P & 0 \\ -h & g_{-1} \end{pmatrix} \times g_0 \in \Aut(U^{-1}) \times \Aut(U^0)$ is such that $$g_0 \cdot (0, v) - (0, v) \cdot \begin{pmatrix} \Id_P & 0 \\ -h & g_{-1} \end{pmatrix} = (0, g_0\cdot v) - (-v \cdot h, v \cdot g_{-1}) = (f_P, f').$$

\subsection{Towards a notion of semistability in $\K^{[-1,0]}(\proj \Lambda)$}\label{ssforpp}
As in \cite{King}, one would like there to exist a \textit{numerical} notion of semistability in $\K_\Lambda = \K^{[-1,0]}(\proj \Lambda)$. Before exploring this idea, let us first recall the notion of virtual semi-invariant, as defined in \cite[Section 6.2]{igusa2009cluster}. 

Let $\theta \in K_0(\proj \Lambda) \simeq \mathbb{Z}^n$, and let $PD(\theta) = \{(\eta^{-1}, \eta^0) \in \mathbb{Z}_{\geq 0}^{2n} \ | \ \eta^0 - \eta^{-1} = \theta\}$. Note if we let let $\theta^{-1} = -(\min(0, \theta_i))_{1 \leq a \leq n}$ and $\theta^0 = (\max{0, \theta_i})_{1 \leq a \leq n}$, then for any $(\eta^{-1}, \eta^0) \in PD(\theta)$, there exist $\gamma \in \mathbb{Z}_{\geq 0}^n$ such that $\eta^i = \theta^i + \gamma$ for $i \in \{-1, 0\}$.  For $\gamma \in \mathbb{Z}_{\geq 0}^{n}$ we define $P(\gamma) = \bigoplus_{i = 1}^n P_i^{\gamma_i}$. Recall that for any $(\eta^{-1}, \eta^0) \in \mathbb{Z}_{\geq 0}^{2n}$ and $\gamma \in \mathbb{Z}_{\geq 0}^{n}$ we have maps $$R(P(\eta^{-1}), P(\eta^0)) \longrightarrow R(P(\eta^{-1})\oplus P(\gamma), P(\eta^0)\oplus P(\gamma))$$

who take any $x \in R(P(\eta^{-1}), P(\eta^0))$ and sends it to $\begin{psmallmatrix} x & 0 \\ 0 & \Id_{P(\gamma)} \end{psmallmatrix}$. We refer to these as \textbf{stabilization maps}. We define the \textbf{virtual representation space} of $\theta$ as the direct limit over $PD(\theta)$

$$R^{vir}(\theta) = \varinjlim_{(\eta^{-1},\eta^0) \in PD(\theta)} R(P(\eta^{-1}), P(\eta^0)).$$

Every $R(\eta^{-1},\eta^0) = R(P(\eta^{-1}), P(\eta^0))$ gives rise to a ring of semi-invarians for the action of the group $G(\eta^{-1},\eta^0) = G(P(\eta^{-1}), P(\eta^0))$. The restriction maps induced by the functions $x \mapsto \begin{psmallmatrix} x & 0 \\ 0 & \Id_P \end{psmallmatrix}$ described above define an inverse system over $PD(\theta^0-\theta^{-1})$ of the rings $SI(R(\theta^{-1},\theta^0))^{G(\theta^{-1},\theta^0)}$. The ring \textbf{virtual semi-invariants} for $\theta \in \mathbb{Z}^n$ is the inverse limit over $PD(\theta)$

$$SI^{vir}(\theta)= \varprojlim_{(\eta^{-1},\eta^0) \in PD(\theta)} SI(R(\eta^{-1},\eta^0))^{G(\eta^{-1},\eta^0)}.$$

A \textbf{virtual semi-invariant} associated to $X$ is element $f$ in $SI^{vir}([X])$. The following proposition tell us that, up to adding enough $\text{\begin{tikzcd}[cramped, sep=tiny]  P \arrow[d, equal]\\ P \end{tikzcd}}$ summands to a given representative of an object $X \in K_\Lambda$, virtual semi-invariants have weights given by $\bar{d} = (d,d)$ for some $d \in \mathbb{Z}^n$. We say that $f \in SI^{vir}(\theta)$ has weight $d \in \mathbb{Z}^n$ when this is the case. 

\begin{proposition}~\cite[Proposition 3.3.3]{igusa2009cluster} Consider $R(X^{-1}, X^0)$ where $X^{-1} = \bigoplus_{i=0}^n P_i^{\theta^{-1}_i}$ and $X^0 = \bigoplus_{i=0}^n P_i^{\theta^0_i}$, with its usual $G(X^{-1}, X^0)$ action. Suppose  there is a non-zero $f \in SI(R)^{G(X^{-1}, X^0), \chi}$  with $\chi = \chi_{\bar{d}}$ for some $\bar{d} = (d^{-1}, d^0) \in \mathbb{Z}^{2n}$. If, for every $1 \leq i \leq n$, both $d^0_i \neq 0$ and $d^{-1}_i \neq 0$ then $d^{-1}_i = d^0_i$.
\end{proposition}

\begin{remark}\label{gemetricss2}
	Let $X \in \K_\Lambda$. Suppose that the exists a virtual semi-invariant $f$ of weight $d \in K_0(\Mod \Lambda)$ such that $f(X) \neq 0$. In particular, there exists a representative $X \simeq \text{\begin{tikzcd}[cramped, sep=small]  X^{-1} \arrow[d, "x"] \\ X^0 \end{tikzcd}}$ such that $f$ defines a semi-invariant for the action of $G(X^{-1}, X^0)$ over $R(X^{-1}, X^0)$ with $f(x) \neq 0$, that is, $x$ is geometrically semistable. We say that $X \in \K_\Lambda$ is geometrically $d$-semistable for $d \in K_0(\Mod \Lambda)$ if there exists a virtual semi-invariant $f$ of weight $d$ such that $f(X) \neq 0$.
\end{remark} 

\cref{detervirtual} implies that the $s(-,M)$ are virtual semi-invariants of weight $\dim M$ for those $\theta \in \mathbb{Z}^n$ such that $\langle \theta, [M]\rangle=0$. Then, if $X \in \K_\Lambda$ is $M$-semistable for $M \in \Mod \Lambda$, it is geometrically $[M]$-semistable. Suppose now that $X \in \K_\Lambda$ is geometrically $d$-semistable By \cref{naive}, we get that 	
\begin{itemize}
	\item[$-$] $\langle (-[X^{-1}], [X^0]), (d,d) \rangle = 0$
	\item[$-$] For any inflation $Y \rightarrowtail X$ in $\C^{[-1,0]}(\proj \Lambda)$, we must have that $$\langle (-[Y^{-1}], [Y^0]), (d,d) \rangle \geq 0.$$
\end{itemize}
Recall that for any inflation $Y \rightarrowtail X$ in $\K_\Lambda$ we can find representatives of $Y$ and $X$ that give an inflation in $\C^{[-1,0]}(\proj \Lambda)$. The following definition is an attempt to summarize these facts. 

\begin{definition}[\textbf{Numerical semistability}]\label{numericalss}
	Let $X \in \K_\Lambda$ and $d \in \mathcal{K}_0(\Mod \Lambda)$. We say that $X$ is numerically $d$-semistable if 
	\begin{enumerate}
		\item $\langle [X], d\rangle = 0$,
		\item For every inflation $ Y \rightarrowtail X$ we have $\langle [Y], d \rangle \geq  0$.
	\end{enumerate}	
\end{definition}

The following theorem links both numerical semistability and $M$-semistability. 

\begin{theorem}\label{Mimplies[M]}
	Let $X \in \K_\Lambda$ and $M \in \Mod\Lambda$ with corresponding class $[M]$ in $\mathcal{K}_0(\Mod\Lambda)$. Suppose $X$ is $M$-semistable, then $X$ is numerically $[M]$-semistable.
\end{theorem}

\begin{proof}
	Since $X$ is $M$-semistable, $s(X,M) \neq 0$. Let $Y \rightarrowtail X$ be an inflation in $\K_\Lambda$. By \cref{definf}, there exists $P \in \proj \Lambda$ such that $Y \rightarrowtail X \oplus \text{\begin{tikzcd}[cramped, sep=tiny]  P \arrow[d, equal]\\ P \end{tikzcd}}$ is an inflation in $\C^{[-1,0]}(\proj \Lambda)$. But $X \oplus \text{\begin{tikzcd}[cramped, sep=tiny]  P \arrow[d, equal]\\ P \end{tikzcd}}$ stills satisfies that the virtual semi-invariant $s(-,M)$ is non-zero. By \cref{naive}, and noting that $\langle (-[X^{-1}],[X^0]), ([M], [M])\rangle = \langle [X], [M] \rangle$ we get the result. 
\end{proof}

Note that if $Y \xrightarrow{f} X$ is a map inside $\K_ \Lambda$, in order for $\text{Cone}(f)$ to lie in $\mathcal{K}$ -which implies that $f$ is an inflation- the map $Y^{-1} \xrightarrow{f^{-1}} X^{-1}$ must be a section. If $X$ is $M$-semistable, we get a commutative square 
\begin{center}
	\begin{tikzcd}
		\Hom(X^0,M) \arrow[d, "\cong"] \arrow[r, "f^{0}"] & \Hom(Y^{0},M) \arrow[d, "y^*"] &  \\
		\Hom(X^{-1},M) \arrow[r, twoheadrightarrow, "f^{-1}"] & \Hom(Y^{-1},M) \arrow[r]& 0
	\end{tikzcd}
\end{center}
This implies that the map $y^*$ is an epimorphism, and thus $$\langle [Y], [X] \rangle = \dim \Hom(Y^0, M) - \dim \Hom(Y^{-1}, M) \geq 0.$$ This gives us a proof of \cref{Mimplies[M]} that does not rely on geometric arguments. 

We end this section with two examples. The first shows that numerical semistability does not necessarily imply geometric semistability in $\K_\Lambda$. The second tells us that the subcategory of objects that are numerically $d$-semistable is not closed under extensions in general.

\begin{example}[\textbf{Numerical semistability does not imply geometric semistability}]\label{ex_numssnotgeoss}
	Consider $\Lambda = \mathbb{C} Q / I$, where $Q$ is the quiver with relations 
	\begin{center}
		\begin{tikzcd}
			1 \arrow[r, bend left, "\alpha"{name= a}] & 2 \arrow[l, bend left, "\beta"{name= b}] \arrow[from=a, to=b, no head, dashed, bend left, shift right = 2.5ex, shorten = 0.95mm] \arrow[from=b, to=a, no head, dashed, bend left, shift right = 2.5ex, shorten = 0.95mm]
		\end{tikzcd}
	\end{center}
	and $I = \langle \alpha \beta, \beta \alpha \rangle$. Consider as well the objects $X_1 = P_1 \xrightarrow{\alpha} P_2$ and $X_2 = P_2 \xrightarrow{\beta} P_1$. We have $\langle [X_1], P_2 \rangle = 0 $ and that $\Hom_{\Lambda}(S_2, P_2) = 0$, and so, the virtual semi-invariant $s(X_1, P_2) = \det( \Hom(P_2, P_2) \xrightarrow{- \circ \alpha} \Hom(P_1, P_2))$ is non-zero, that is, $X_1$ is $P_2$-semistable. Likewise, $X_2$ is $P_1$-semistable since $s(X_2, P_1) \neq 0$. Consider now $X = X_1 \oplus X_2 \in \Hom(P_1 \oplus P_2, P_1 \oplus P_2)$ with representative $ x  = \begin{pmatrix} 0 & \beta \\ \alpha & 0 \end{pmatrix} \in R(P_1 \oplus P_2, P_1 \oplus P_2)$. Note that $X = $\begin{tikzcd}[cramped, sep=small] P_1 \oplus P_2 \arrow[d, "x"] \\ P_1 \oplus P_2 \end{tikzcd} satisfies the two properties of \cref{numericalss} for the vector $d = [P_1] = [P_2] =(1,1)$. Since $[X] = (0,0)$, it satisfies (1). On the other hand, $X_1$, $X_2$, $0 \rightarrow P_1$ and $0 \rightarrow P_1$ are the only possible indecomposable direct summands of objects that are the source of inflations into $X$, and they all satisfy condition (2). Thus $X$ is numerically $(1,1)$-semistable. We will show that there is no virtual semi-invariant $f$ of weight $(1,1)$ such that $f(X) \neq 0$. 
	
	Consider $$x' = \left( \begin{array}{cc|cc}  0 & \mathbf{0}_{1 \times n} & 1 & \mathbf{0}_{1 \times n} \\ \mathbf{0}_{n \times 1} & \Id_{n} & \mathbf{0}_{1 \times n} & \mathbf{0}_{n} \\ \hline 1 & \mathbf{0}_{1 \times n} & 0 & \mathbf{0}_{1 \times n} \\ \mathbf{0}_{n \times 1} & \mathbf{0}_{n} & \mathbf{0}_{n \times 1} & \Id_{n}\end{array} \right)$$
	the image of $x$ by the stabilization map $R(P_1 \oplus P_2, P_1 \oplus P_2) \rightarrow R(P_1^{n+1} \oplus P_2^{n+1} , P_1^{n+1}  \oplus P_2^{n+1})$ for $n\gg 0$. Let 
	\begin{gather*}
		R_n = R(P_1^{n+1} \oplus P_2^{n+1} , P_1^{n+1}  \oplus P_2^{n+1}) = \left( \begin{array}{cc}
			M_{n+1}(\mathbb{C}) \cdot \Id_{P_1} & M_{n+1}(\mathbb{C}) \cdot \beta \\ 
			M_{n+1}(\mathbb{C}) \cdot \alpha & M_{n+1}(\mathbb{C}) \cdot \Id_{P_2} \end{array}\right),
	\end{gather*}
	\begin{gather*}
		G_n = G(P_1^{n+1} \oplus P_2^{n+1} , P_1^{n+1}  \oplus P_2^{n+1}) = \\ \left( \begin{array}{cc}
			GL_{n+1}(\mathbb{C}) \cdot \Id_{P_1} & M_{n+1}(\mathbb{C}) \cdot \beta \\ 
			M_{n+1}(\mathbb{C}) \cdot \alpha & GL_{n+1}(\mathbb{C}) \cdot \Id_{P_2} \end{array}\right)^{op} \times \left( \begin{array}{cc}
			GL_{n+1}(\mathbb{C}) \cdot \Id_{P_1} & M_{n+1}(\mathbb{C}) \cdot \beta \\ 
			M_{n+1}(\mathbb{C}) \cdot \alpha & GL_{n+1}(\mathbb{C}) \cdot \Id_{P_2} \end{array}\right), 
	\end{gather*}
	\begin{gather*}
		U_n = \left( \begin{array}{cc}
			\Id_n \cdot \Id_{P_1} & M_{n+1}(\mathbb{C}) \cdot \beta \\ 
			M_{n+1}(\mathbb{C}) \cdot \alpha & \Id_n \cdot \Id_{P_2} \end{array}\right)^{op} \times \left( \begin{array}{cc}
			\Id_n \cdot \Id_{P_1} & M_{n+1}(\mathbb{C}) \cdot \beta \\ 
			M_{n+1}(\mathbb{C}) \cdot \alpha & \Id_n \cdot \Id_{P_2} \end{array}\right), 
	\end{gather*}
	\begin{gather*}
		(G_n)_{red} = \\ \left( \begin{array}{cc}
			GL_{n+1}(\mathbb{C}) \cdot \Id_{P_1} & 0\\ 
			0& GL_{n+1}(\mathbb{C}) \cdot \Id_{P_2} \end{array}\right)^{op} \times \left( \begin{array}{cc}
			GL_{n+1}(\mathbb{C}) \cdot \Id_{P_1} & 0\\ 
			0 & GL_{n+1}(\mathbb{C}) \cdot \Id_{P_2} \end{array}\right),
	\end{gather*}
	where $U_n$ is the unipotent radical of $G_n = (G_n)_{red} \rtimes U_n$. The group $G_n$ acts on $R_n$ in the following way. For every $g = \begin{pmatrix} X & Y \\ Z & W \end{pmatrix} \times \begin{pmatrix} X' & Y' \\ Z' & W' \end{pmatrix} \in G_n$ and $y = \begin{pmatrix} A & B \\ C & D \end{pmatrix} \in R_n$ we have that
	$$g \cdot y = \begin{pmatrix} X'AX & X'AY +X'BW +Y'DW \\ Z'AX+W'CX+W'DZ & W'DW \end{pmatrix}.$$
	Any $G_n$-semi-invariant $f$ on $R_n$ must be a $U_n$-invariant function. In particular, $f$ must be invariant for the action of the subgroup $V = \Id_{2n+2} \times \begin{pmatrix} \Id_{n+1} & M_{n+1}(\mathbb{C}) \cdot \beta \\ 0 & \Id_{n+1}\end{pmatrix}$. Since $\begin{pmatrix} \Id_{n+1} & Y' \\ 0 & \Id_{n+1} \end{pmatrix} \cdot \begin{pmatrix} A & B \\ C & D \end{pmatrix} = \begin{pmatrix} A & B + Y'D \\ C & D \end{pmatrix}$, then the ring of invariants $\mathbb{C}[R_n]^{V} = \mathbb{C}[A,B,C,D]^{V}$ is isomorphic to $\mathbb{C}[A,C] \otimes \mathbb{C}[B,D]^{V'}$, where $\mathbb{C}[B,D]^{V'}$ is the ring of invariants of the action of $V' = \begin{pmatrix} \Id_{n+1} & M_{n+1}(\mathbb{C}) \\ 0 & \Id_{n+1}\end{pmatrix}$ over the set of $2n+2$ by $n+1$ matrices $\begin{pmatrix} B \\D \end{pmatrix}$ given by left multiplication.
	In \cite[Section 4.2]{pommerening1987ordered}, the author explicity describes a basis of the invariant functions on the variety of $N\times M$ matrices under the action of certain unipotent subgroups of $GL_N(\mathbb{C})$ by left multiplication. Applying these results to our particular case, we obtain that $\mathbb{C}[B,D]^{V'} = \mathbb{C}[D]$, and thus $\mathbb{C}[R_n]^{V} = \mathbb{C}[A,C,D]$. A similar argument shows that the ring of invariants on $R_n$ by the action of $ H = \Id_{2n+2} \times \begin{pmatrix} \Id_{n+1} & 0 \\ M_{n+1}(\mathbb{C}) \cdot \alpha & \Id_{n+1}\end{pmatrix}$ is $\mathbb{C}[A,B,D]$, which in turn implies that $\mathbb{C}(R_n)^{V \cdot H} = \mathbb{C}[A,D]$, where $V \cdot H = \Id_{2n+2} \times \begin{pmatrix} \Id_{n+1} & M_{n+1}(\mathbb{C}) \cdot \beta \\ M_{n+1}(\mathbb{C}) \cdot \alpha & \Id_{n+1}\end{pmatrix}$.  Moreover, the functions in $\mathbb{C}[A,D]$ are also invariant for the action of $\begin{pmatrix} \Id_{n+1} & M_{n+1}(\mathbb{C})\cdot \beta\\ M_{n+1}(\mathbb{C})\cdot \alpha & \Id_{n+1}\end{pmatrix}^{op} \times \Id_{2n+2}$, which gives us that $\mathbb{C}[R_n]^{U_n} = \mathbb{C}[A,D]$.
	The only $(G_n)_{red}$ semi-invariants of weight $(1,1)$ on $\mathbb{C}[A,D]$ are the functions $f_k(A,D) = k \cdot \det(A) \cdot \det(D)$ for $k \in \mathbb{C}^*$ \cite[Theorem 4.4.4]{derksen2015computational}, and for all of them, $f_k(x') = 0$, since $A_{x'} = D_{x'}$ are not of full rank. That is, $X$ is not geometrically $(1,1)$-semistable.

\end{example}

\begin{example}[\textbf{Numerical semistability is not closed under extensions}]\label{ex_numericalnotgood}
	Consider again the quiver of the last example. As we have seen, both $X_1 = P_1 \xrightarrow{\alpha} P_2$ and $X_2 = P_2 \xrightarrow{\beta} P_1$ are numerically $(1,1)$-semistable. In $\K_\Lambda$, the following sequence is a conflation
	\begin{center}
		\begin{tikzcd}
			P_1 \arrow[r, equal] \arrow[d, "\alpha"] & P_1 \arrow[d, "0"] \arrow[r, "\alpha"] & P_2 \arrow[d, "\beta"] \\
			P_2 \arrow[r, "\beta"] &  P_1 \arrow[r, equal] & P_1
		\end{tikzcd}
	\end{center}
	However, $P_1 \oplus P_1[1]$ is not $(1,1)$-numerically semistable. Indeed $\langle -[P_1], (1,1) \rangle = -1 \leq 0$. 
\end{example}

\bibliographystyle{alpha}
\bibliography{biblio}

\end{document}